\documentclass[12pt,reqno]{amsart}
\usepackage[scale=0.86]{geometry}
\usepackage{geometry} 
 \usepackage{color}
\usepackage{graphicx}
\usepackage{bbm,enumerate}
\usepackage{mathrsfs}
\theoremstyle{plain}
\newtheorem{theorem}{Theorem}
\numberwithin{theorem}{section}
\newtheorem{lemma}[theorem]{Lemma}
\newtheorem{example}[theorem]{Example}

\newtheorem{corollary}[theorem]{Corollary}
\newtheorem{proposition}[theorem]{Proposition}

\theoremstyle{definition}

\newtheorem{remark}[theorem]{Remark}

\allowdisplaybreaks
\numberwithin{equation}{section}
\title[Continuous-state branching processes with collisions]{Continuous-state branching processes with collisions: first passage times and duality}
\author{Cl\'ement Foucart}\thanks{foucart@math.univ-paris13.fr, Universit\'e Paris 13,  Laboratoire Analyse, G\'eom\'etrie \& Applications UMR 7539 Institut Galil\'ee}
\author{Matija Vidmar}\thanks{matija.vidmar@fmf.uni-lj.si, Department of Mathematics, Faculty of Mathematics and Physics, University of Ljubljana.}

\usepackage[pagebackref,colorlinks,citecolor=Plum,urlcolor=Periwinkle,linkcolor=DarkOrchid]{hyperref}
\usepackage[dvipsnames]{xcolor}

\newcommand{\dd}{\mathrm{d}}

\newcommand{\PP}{\mathbb{P}}

\newcommand\FF{\mathcal{F}}
\newcommand\GG{\mathcal{G}}

\newcommand\EE{\mathbb{E}}
\newcommand\CBC{$\mathrm{CBC}(\Sigma,\Psi)$ }

\newcommand{\ddr}{\mathrm{d}}
\newcommand{\mc}[1]{\ensuremath{\mathscr{#1}}}

\makeatletter
\@namedef{subjclassname@2020}{%
  \textup{2020} Mathematics Subject Classification}
\makeatother
\begin{document}

\begin{abstract}
We introduce a class of one-dimensional positive Markov processes generalizing continuous-state branching processes (CBs), by taking into account a phenomenon of random collisions. Besides branching, characterized by a general mechanism $\Psi$, at a constant rate in time two particles are sampled uniformly in the population, collide and leave a mass of particles governed by a (sub)critical mechanism $\Sigma$. Such CB processes with collisions (CBCs) are shown to be the only Feller processes without negative jumps satisfying a Laplace duality relationship with one-dimensional diffusions on the half-line.  This generalizes the duality observed for logistic CBs in Foucart \cite{MR3940763}. Via time-change, CBCs are also related to an auxiliary class of Markov processes, called CB processes with spectrally positive migration (CBMs), recently introduced in Vidmar \cite{vidmar2021continuousstate}. We find necessary and sufficient conditions for the boundaries $0$ or $\infty$ to be attracting and for a limiting distribution to exist. The Laplace transform of the latter is provided. Under the assumption that the CBC process does not explode, the Laplace transforms of the first passage times below arbitrary levels are represented with the help of the solution of  a second-order differential equation, whose coefficients express in terms of the L\'evy-Khintchine functions $\Sigma$ and $\Psi$. Sufficient conditions for non-explosion are given. 
\end{abstract}

\maketitle
%

%
\section{Introduction}
\subsection{Motivation}
Imagine a set of particles evolving according to the  following two rules: branching occurs at random, and whenever a particle splits (dies), it begets $k\in \mathbb{Z}_+:=\{0,1,\ldots\}$ new particles  with probability  $p_b(k)$. In the time interval $(t,t+\Delta t)$, the probability for a branching of any given particle to occur is of order $b\Delta t+o(\Delta t)$. Multiple branching events occur in this time interval with probability $o(\Delta t)$.
Particles are also allowed to collide. Whenever a collision between two particles occurs, the pair is replaced by $k+1$ new particles with probability $p_c(k)$, $k\in \mathbb{Z}_+$. The average number of particles left after a collision is assumed to be less than or equal to two: $\sum_{k\in \mathbb{Z}_+}kp_c(k)\leq 1$. The probability of a single collision to occur for any given pair of individuals during $(t,t+\Delta t)$ is of order $c\Delta t+o(\Delta t)$, and for multiple collisions it is $o(\Delta t)$. All collisions and branching events are assumed to occur independently from each other. The infinitesimal generator $\mathcal{L}$ of the continuous-time Markov chain recording the number of particles then takes the form
\[\mathcal{L}f(n):=c\binom{n}{2}\sum_{k=0}^{\infty}\left(f(n+k-1)-f(n)\right)p_c(k)+bn\sum_{k=0}^{\infty}\left(f(n+k-1)-f(n)\right)p_b(k),\] for bounded $f:\mathbb{Z}_+\to \mathbb{R}$ and $n\in \mathbb{Z}_+$, with $\binom{n}{2}=0$ if $n\in\{0,1\}$. Processes with generator $\mathcal{L}$ have  been considered by Chen et al. in  \cite{zbMATH06022670,zbMATH07297549,MR2122478},
and they recently reappeared in the works of  Gonz\'alez-Casanova et al. \cite{zbMATH07458586} and Berzunza Ojeda and Pardo \cite{ojeda2020branching}. The study in the two latter works makes use of a moment-duality relationship with some generalized Wright-Fisher diffusions. 

Our aim is to introduce and initiate the study of the continuous-state space counterparts of these processes and to shed light on some new duality relationships with certain diffusions on the positive half-line whose drift and diffusion functions are of the Lévy-Khintchine form. We adopt here the terminology \textit{collision} coined in \cite{MR2122478}. The terms negative cooperation or pairwise interactions are also used in the literature. 

\subsection{A stochastic equation} In order to explain how collisions are encoded in the continuous-state space setting we shall work in Dawson and Li's framework of SDEs, see \cite{DawsonLi}, and consider the following generalisation of the stochastic equation solved by a continuous-state branching  process (CB) with branching mechanism $\Psi$ (a CB($\Psi$)), for a starting value $z\in [0,\infty)$, 
\begin{align} 
Z_t= \ z &+\sigma\int_{0}^{t}\!\!\sqrt{Z_s}\ddr B_s+b\int_{0}^{t}Z_{s}\ddr s 
 +\int_{0}^{t}\!\!\int_{0}^{Z_{s-}}\!\!\int_0^{1}  h\bar{\mathcal{N}}(\ddr s,\ddr u,\ddr h)+\int_{0}^{t}\!\!\int_0^{Z_{s-}}\!\!\int_1^\infty  h\mathcal{N}(\ddr s,\ddr u,\ddr h)   \nonumber\\
&+a\int_{0}^{t}\!\!Z_s \ddr W_s-\frac{c}{2}\int_{0}^{t}\!\!Z_s^{2}\ddr s+\int_{0}^{t}\! \int_0^{Z_{s-}}\!\!\int_0^{Z_{s-}}\!\!\int_0^\infty h\bar{\mathcal{M}}(\ddr s,\ddr u_1,\ddr u_2,\ddr h). \label{cbcequation}
\end{align}
In \eqref{cbcequation} we adhere to the convention that the lower delimiters are excluded from the integration, while the upper delimiters are included (except for $\infty$, which is of course excluded). The individual ingredients of \eqref{cbcequation} are specified as follows. 

The first line in Eq.~\eqref{cbcequation} represents the branching dynamics and forms the classical stochastic equation solved by a CB (without the need for a finite first moment, see e.g. Ji and Li \cite[Theorem 3.1]{zbMATH07189529}): the parameters $b\in \mathbb{R}$, $\sigma\in [0,\infty)$ are the diffusive coefficients, $B$ is a Brownian motion, finally $\mathcal{N}(\dd s,\dd u,\dd h)$ is an independent Poisson random measure on $[0,\infty)^2\times (0,\infty)$ with intensity $\ddr s\ddr u\pi(\ddr h)$, $\pi$ being a measure on $(0,\infty)$ satisfying $\int_0^\infty 1\wedge h^2 \pi(\ddr h)<\infty$ (a L\'evy measure) and  $\bar{\mathcal{N}}$ stands for the compensated random measure, $\bar{\mathcal{N}}(\ddr s,\ddr u,\ddr h):=\mathcal{N}(\ddr s,\ddr u,\ddr h)-\ddr s \ddr u \pi(\ddr h)$.\footnote{Also in the continuation of this text $\bar M$  will always designate the compensated measure of $M$, whatever the Poisson random measure $M$ may be.} Heuristically, prior to an atom of time $t$ of $\mathcal{N}$, an individual $u$ is chosen uniformly in $[0,Z_{t-}]$ and reproduces or dies. The branching part is governed by the Lévy-Khintchine function \begin{equation}
\label{branchingmechanism}\Psi(x):=\frac{\sigma^2}{2}x^2-b x+\int_{0}^{\infty}\left(e^{-xh}-1+xh\mathbbm{1}_{\{h\leq 1\}}\right)\pi(\ddr h),\quad x\in [0,\infty).
\end{equation}

The second line in Eq.~\eqref{cbcequation} represents  \textit{collisions}: again the parameters $a\in [0,\infty)$, $c\in [0,\infty)$ are the diffusive coefficients, $W$ is a Brownian motion, finally $\mathcal{M}(\dd s,\dd u_1,\dd u_2,\dd h)$ is an independent Poisson random measure on $[0,\infty)^3\times (0,\infty)$ with intensity $\ddr s \ddr u_1  \ddr u_2 \eta(\ddr h)$, $\eta$ being a L\'evy measure on $(0,\infty)$ satisfying $\int_0^{\infty} h\wedge h^2 \eta(\ddr h)<\infty$. The stochastic drivers $(B,\mathcal{N})$ and $(W,\mathcal{M})$  are defined under a common probability $\PP$ and are independent of one-another, i.e. collisions are independent of the branching. Heuristically, prior to an atom of time $t$ of $\mathcal{M}$, two individuals  $u_1$ and $u_2$ are picked uniformly in the population, they collide and are replaced by an amount $h$ of new individuals. The collision part is governed by the Lévy-Khintchine function \begin{equation}
\label{collisionmechanism}\Sigma(x):=\frac{a^2}{2}x^2+\frac{c}{2} x+\int_{0}^{\infty}\left(e^{-xh}-1+xh\right)\eta(\ddr h),\quad x\in [0,\infty).
\end{equation}
We assume the collision mechanism $\Sigma$ is subcritical or critical (i.e. $\Sigma'(0+)=\frac{c}{2}\geq 0$,  which dovetails with  $\sum_{k\in \mathbb{Z}_+}kp_c(k)\leq 1$ above\footnote{Note the sign convention on $c$: it is $-c/2$ that is the ``linear drift'', cf. \eqref{genLevy-2} to follow.\label{footnote:drift}}) but not zero. Thus collisions are either diminishing the number of individuals or keeping it the same on average. One might expect some phenomenon of regulation of the population size when the latter reaches large values. Collisions may for instance prevent or not the growth of the population induced  by supercritical branching dynamics. 

We stress that the compensated versions of the Poisson random measures $\mathcal{M}(\dd s,\dd u_1,\dd u_2,\dd h)$ and $\mathcal{N}(\dd s,\dd u,\dd h)$, the latter  on $h\in (0,1]$, are used in \eqref{cbcequation}. The two stochastic integrals with respect to the independent Brownian motions in \eqref{cbcequation}  can be rewritten respectively as follows: \[
\int_0^t\sqrt{Z_s} \ddr B_s=\int_0^t\int_{[0,Z_s]}\tilde B(\ddr s, \ddr u)  \text{ and } \int_0^tZ_s \ddr W_s=\int_0^t \int_{[0,Z_s]\times [0,Z_s]} \tilde W(\ddr s, \ddr u_1,\ddr u_2),\]
with $\tilde B(\ddr s, \ddr u)$ and $\tilde W(\ddr s, \ddr u_1,\ddr u_2)$  independent Gaussian time-space white noises  on $(0,\infty)\times (0,\infty)$ and $(0,\infty)\times (0,\infty)^2$ based on the Lebesgue measures $\ddr s  \ddr u$ and $\ddr s \ddr u_1  \ddr u_2$, respectively. This allows one to interpret also both diffusive parts  in terms of branching and collision.  We refer to Li and Ma \cite[page 940]{ChunhuaLi} for the  representation of the Feller's branching diffusion part with the white-noise $\tilde{B}(\ddr s,\ddr u)$, see also Pardoux \cite[Section 4.1]{MR3496029} and El Karoui and M\'el\'eard \cite{zbMATH04137113} for a more general framework.

\subsection{Highlights}   The structure of continuous-state branching processes with collisions (CBCs) is  of specific theoretical interest, since, as we shall see, they form a class of  Markov processes with no negative jumps, whose long-term behavior and first passage times can be linked in a dual way to those of certain one-dimensional diffusions. We give now a brief panorama of this.

Once unique existence of $\mathrm{CBC}(\Sigma,\Psi)$,  minimal solution to the stochastic equation \eqref{cbcequation}, has been established, we study three classical problems for this process. First, in Theorem \ref{attractiveboundaries}, we give necessary and sufficient conditions on the branching and collision mechanisms for the \CBC process to converge as time goes to infinity towards its boundaries $0$ or $\infty$ with positive probability (attracting boundaries). We next work under the assumption that the \CBC process does not explode (for which sufficient conditions will be given in Proposition~\ref{propositionsuffcond}) and  obtain,  see Theorem \ref{firstpassagetimestheorem},  a representation of the Laplace transform, at argument $\theta\in (0,\infty)$, of the first passage time below a given level  with the help of an increasing positive function $h_\theta$, solution $h\in C^2((0,\infty))$ to the 
second-order linear differential equation
\begin{equation}\label{generatorG} \mathscr{G}h:=\Sigma h''+(\Sigma'+\Psi)h'=\theta h \text{ on }(0,\infty).
\end{equation} Fundamental results of Feller, see \cite{zbMATH03073345,zbMATH03094597,zbMATH03108758}, on equations of the type  \eqref{generatorG} apply and allow one for instance to understand $h$ in terms of the first passage time of a diffusion $V$ with generator $\mathscr{G}$. 
 In particular, we identify in Theorem \ref{LTextinctiontheorem} the law of the extinction time of the \CBC started from $z$ with that of the explosion time of $V$ started from an independent exponential random variable with parameter $z$. Lastly, we find necessary and sufficient conditions on the mechanisms $\Sigma$ and $\Psi$ for the \CBC to have a limiting distribution, in which case we procure an explicit formula for its Laplace transform, see Theorem \ref{stationarydisttheorem}.

An important feature of CBCs lies in the fact that their generator $\mathscr{L}$ satisfies a Laplace duality  (also called exponential duality)  with the diffusion generator $\mathscr{A}$ given by, for $g\in C^2([0,\infty))$,  \begin{equation}\label{generatorA} \mathscr{A}g:=\Sigma g''-\Psi g'\text{  on $[0,\infty)$};
\end{equation} 
namely,  for  $\{x,z\}\subset [0,\infty)$,
\[\mathscr{L}_ze^{-xz}=\mathscr{A}_xe^{-xz}\]
(the subscripts in the operators indicate which variable they are acting on). When the \CBC process $Z:=(Z_t,t\geq 0)$ does not explode we will establish that also its semigroup satisfies a Laplace duality relationship; to wit, for $\{t,x,z\}\subset [0,\infty)$: 
\[\mathbb{E}_z\big[e^{-xZ_t}\big]=\mathbb{E}_x\big[e^{-zU_t}\big],\]  $U$ being the diffusion on $[0,\infty)$ with $0$ an absorbing state and generator $\mathscr{A}$ ($U$, as it emerges, does not explode [either]), see  Proposition \ref{Laplacedualsemigroup}.  Conversely, it will be established  in Theorem~\ref{thm:laplace-converse} that under a mild assumption on the domains of their generators, CBCs are the only positive Feller processes with no negative jumps and zero an absorbing state, whose generators are in Laplace duality with those of one-dimensional diffusions.

The generator $\mathscr{A}$ of \eqref{generatorA} in turn is in so-called Siegmund duality with the generator $\mathscr{G}$ of \eqref{generatorG}. Under certain conditions, which shall be specified later on, and which entail that $Z$ does not explode, we get it again at the level of the semigroups: for $t\in[ 0,\infty)$ and $\{x,y\} \subset (0,\infty)$,
\[\mathbb{P}_x(U_t<y)=\mathbb{P}_y(V_t>x),\]
see  Siegmund \cite{MR0431386}, Cox and R\"osler \cite{MR724061}, \cite[Section 6]{foucart2021local} and the forthcoming  Proposition \ref{Siegmundualsemigroup}. It will emerge (Remark~\ref{Remark0notexit}) that under the assumption of non-explosion of the CBC  the boundary $0$, like $\infty$, of $U$ is also never regular, and in that case there is therefore in fact no need to stipulate  boundary conditions for $U$ and $V$.

The preceding duality relationships explain somehow why the study of the boundaries of $Z$ is linked to those of the processes $U$ and $V$ (for which the general theory of diffusions applies).  At the level of the semigroups they will actually be used mainly for studying the existence of the limiting distribution of $Z$ and for its  characterization. At the level of the generators they lie however at the core of our study for all CBCs, and are summarized by the following diagram, in which,  for the reader's convenience, we also note the corresponding processes:
\begin{equation*} (Z,\mathscr{L}) \overset{\textbf{Laplace dual}}{\longleftrightarrow} (U,\mathscr{A}) \overset{\textbf{Siegmund dual}}{\longleftrightarrow} (V,\mathscr{G}).
\end{equation*} 
The reader is referred to Kurt and Jansen's survey \cite{duality} for a recent general account of duality. 

\subsection{Literature overview and available examples} Attention has recently been paid to the role of duality in the study of eigenfunctions of generators, see for instance  Griffiths \cite{zbMATH06387922}, Foucart and M\"ohle \cite{zbMATH07544388} and Redig and Sau \cite{zbMATH07202014}.   
We should also like to point out that there is a relatively vast and developing literature on exit problems of  Markov processes with one-sided jumps to which our study can be connected. In this vein we may mention, restricting to continuous space and time, Duhalde et al. \cite{Duhalde} for CBs with immigration (CBIs), Kuznetsov et al. \cite{kkr} for spectrally negative L\'evy processes, Borovkov and Novikov \cite{zbMATH05344878} and Patie \cite{zbMATH02192577} for generalized Ornstein-Uhlenbeck processes, Patie \cite{zbMATH05547564,pierre}  and Vidmar \cite{vidmar2021exit} for positive self-similar Markov processes, see also Landriault et al. \cite{landriault_li_zhang_2017} and Avram et al. \cite{avram_li_li_2021} for some general drawdown/drawup results.

A few examples of CBCs have already appeared in the literature. The pure drift collision mechanism, namely the case $\Sigma(x)=\frac{c}{2}x$, $x\in [0,\infty)$, corresponds to logistic CBs, see Lambert  \cite{MR2134113} and Foucart \cite{MR3940763,foucart2021local}. Duality relationships with the processes $U$ and $V$ were observed in the two latter works, however their role in the problem of characterizing first passage times of the logistic CB $Z$ was not understood therein.  In the pure quadratic collision mechanism, when $\Sigma(x)=\frac{a^2}{2}x^2$,  $x\in [0,\infty)$, CBCs match with CBs in Brownian environment, see Palau and Pardo \cite{zbMATH06684580} and He et al. \cite{zbMATH06963718}. The case $\Sigma(x)=\frac{a^2}{2}x^2+\frac{c}{2}x$,  $x\in [0,\infty)$, has also been recently studied by Leman and Pardo \cite{zbMATH07470626}, by adapting Lambert's method in \cite{MR2134113}, which relies on the study of some Ricatti-type nonlinear equations. 
The duality was not used in the latter article and will simplify the study for us at several levels, especially for the limiting distribution. Lastly, the pure continuous-state collision process, for which $\Psi= 0$ and the first line in Eq.~\eqref{cbcequation} vanishes, corresponds to a polynomial CB process, defined in Li \cite{zbMATH07107490}, with power $\theta=2$. In this case \cite[Theorem 1.8]{zbMATH07107490} ensures that there is no extinction in finite time of the process (i.e. $0$ is inaccessible).

In the subcritical collision case, i.e. $c>0$, the process behaves in many aspects as the logistic CB. The critical collision case when $c=0$ is  however  more involved to study and many different new behaviors in comparison to the subcritical one may exist. This is merely due to the fact that fluctuations of the martingale part in the second line of Eq.~\eqref{cbcequation} are now involved and not the deterministic quadratic drift. 
We will not address here the complete classification of the boundary $\infty$ of CBCs. It seems indeed to  require a study of its own since all types (natural, entrance, exit, regular) may occur. We may refer the interested  reader however to \cite{MR3940763} where the case of logistic CBs is treated.

\subsection{Article structure}
Main results  are gathered in Section~\ref{sectionmainresult}, their proofs  deferred to the continuation of the text. Specifically, in Section~\ref{ProofofTheorem1} we study the stochastic equation \eqref{cbcequation} and show that its solution is related,  via Lamperti time-change, to a class of processes, called CB processes with spectrally positive migration (CBMs) in \cite{vidmar2021continuousstate}. Then, in Section~\ref{sectionproofattraction} we study the attraction of the boundaries, in Section~\ref{sectionstudyfirstpassagetime} the first passage times, and in Section~\ref{sectionduality} the duality relationships and the limiting distribution, as indicated above. 
%

\section{Main results}\label{sectionmainresult}
\subsection{Introduction of CBCs and first properties}
For the Lévy-Khintchine function $\Psi$ we denote by $\mathrm{L}^{\Psi}$ the infinitesimal generator of a spectrally positive Lévy process with Laplace exponent $\Psi$. It acts on a $C^2_0(\mathbb{R})$ (i.e. twice continuously differentiable $f:\mathbb{R}\to \mathbb{R}$ with $f,f',f''$ all vanishing at infinity)  function $f$ as follows:
\begin{equation}\label{genLevy}\mathrm{L}^\Psi f(z)=\frac{\sigma^2}{2} f''(z) +b f'(z)+\int_{0}^{\infty}\left(f(z+h)-f(z)-hf'(z)\mathbbm{1}_{(0,1]}(h)\right)\pi(\ddr h),\quad z\in \mathbb{R}.
\end{equation}
Analogously for $\Sigma$ we have
\begin{equation}\label{genLevy-2}\mathrm{L}^\Sigma f(z)=\frac{a^2}{2} f''(z) -\frac{c}{2} f'(z)+\int_{0}^{\infty}\left(f(z+h)-f(z)-hf'(z)\right)\eta(\ddr h),\quad z\in \mathbb{R}.
\end{equation}
It is clear that if $f:I\to\mathbb{R}$ is of class $C^2_b(I)$ ($b$ stands for $f$ being bounded\footnote{but its first and second derivative \emph{need not be bounded}!}), defined (only) on some interval $I$ of $\mathbb{R}$ unbounded above, then $\mathrm{L}^\Psi $ and $\mathrm{L}^\Sigma $ are  (still) naturally defined on $I$ by the right-hand sides of the preceding displays. We take this for granted in the continuation of the text. 

Below, for notions such as adaptedness, martingale etc. we work with  the augmented natural filtration   $\mathcal{F}$  of $(W,B,\mathcal{N},\mathcal{M})$, unless explicitly noted otherwise.
 
\begin{theorem} \label{theorem:cbc-sde-contruction}
For each starting value $z\in [0,\infty)$ there exists an a.s.\! unique $[0,\infty]$-valued c\`adl\`ag adapted process $Z=(Z_t,t\geq 0)$ such that, setting
\begin{equation}\label{eq.upwards-def}
\zeta_n^+:=\inf\{t\in [0,\infty):Z_t\geq  n\}\text{ for } n\in [0,\infty) \text{ and then }\zeta_\infty:=\lim_{n\to\infty}\zeta_n^+,
\end{equation}
the following two properties hold: 
\begin{enumerate}[(i)]
\item\label{theorem:cbc-sde-contruction:ii} \emph{$Z$ satisfies  \eqref{cbcequation} up to  first explosion}, more precisely,  $\zeta_\infty>0$ and $Z$ satisfies \eqref{cbcequation} for $t\in [0,\zeta_\infty)$ a.s.;
\item\label{theorem:cbc-sde-contruction:i} \emph{$Z$ is absorbed at $\infty$ after its first explosion}, that is to say,  $Z=\infty$  on $[\zeta_\infty,\infty)$ a.s..
%
%
\end{enumerate} 
The law of the process $Z$ is uniquely determined by the triplet $(\Sigma,\Psi,z)$. Furthermore, the process $Z$  is a.s. without negative jumps, has $0$ as an absorbing state, is quasi-left continuous and strong Markov, finally,  for all $f\in C^2_b([0,\infty))$, setting 
\begin{equation}\label{genZ} \mathscr{L}f(z):=z^2\mathrm{L}^{\Sigma}f(z)+z\mathrm{L}^{\Psi}f(z),\quad z\in [0,\infty),\end{equation} then for all $\alpha\in [0,\infty)$  and for all $n\in [0,\infty)$, the  process 
\begin{equation}\label{cbc-construction:i:a}
f(Z_{t\land \zeta_n^+})e^{-\alpha (t\land \zeta_n^+)}-\int_0^{t\land \zeta_n^+}e^{-\alpha s}(\mathscr{L}f(Z_s)-\alpha f(Z_s))\dd s,\quad t\in [0,\infty),
\end{equation}
is a local martingale.
\end{theorem}

%
We call the process $Z$ the (minimal) $\mathrm{CBC}(\Sigma,\Psi)$. We shall not have occasion to deal with extensions of the minimal process in this paper, so the qualification ``minimal'' will be largely omitted, except for emphasis. We retain the notation introduced in the theorem, namely \eqref{eq.upwards-def}-\eqref{genZ},   and set further 
\begin{equation}\label{eq.first-passage.def}
\zeta_\mathsf{a}^-:=\inf\{t\in [0,\infty): Z_t\leq \mathsf{a}\},\quad \mathsf{a}\in [0,\infty).
\end{equation}
 In order to stress the initial value $z\in [0,\infty)$ we write $\PP_z$ instead of just $\PP$ and correspondingly $\mathbb{E}_z$ rather than just $\mathbb{E}$, being somewhat lax about holding $z$ fixed or variable. In view of the martingale claim surrounding \eqref{cbc-construction:i:a} (with $\alpha=0$) we call $\mathscr{L}$ the generator of the \CBC process $Z$. Glancing at \eqref{genZ}, likewise as for $\mathrm{L}^\Psi $ and $\mathrm{L}^\Sigma$, so too may we (and shall) consider $\mathscr{L}$ as being capable of taking as input any $f:I\to\mathbb{R}$ that is $C^2_b(I)$, defined (only) on some interval $I$ of $[0,\infty)$ unbounded above, returning in this case a map defined on $I$ according to the right-hand side of \eqref{genZ}.

The branching mechanism $\Psi$ may be of two fundamentally different forms: 
either $\Psi(x)\leq 0$ for all $x\geq 0$ so that $-\Psi$ is the Laplace exponent of a subordinator (we include under this designation the constant process) and in which case we say that we are in the subordinator case; or $\Psi(x)> 0$ for some $x>0$ so that $\Psi$ is the Laplace exponent of a spectrally positive L\'evy process (in the narrow sense) or of a negative linear drift. In the subordinator case,  no particle dies after a branching event and the pure CB$(\Psi)$ has non-decreasing sample paths. Unsurprisingly this case is (can be) a little singular also in the present context as the next proposition demonstrates.
\begin{proposition}\label{propositionstatespace} 
Put 
\begin{equation}\label{eq:z-star}
z^*:=\left(\limsup_{u\to\infty} \frac{-\Psi(u)}{\Sigma(u)}\right)\vee 0<\infty.
\end{equation} If $z^*>0$, and if $Z_0=z>z^*$ (resp. $Z_0=z\leq z^*$) a.s., then a.s. $Z_t> z^*$ for all $t\in [0,\infty)$ (resp.  $Z_t\geq z$ for all $t\in [0,\infty)$ and $\liminf_{t\to\infty}Z_t\geq z^*$). Furthermore, one has $z^*>0$ if and only if $\Sigma(x)/x \underset{x\rightarrow \infty}{\longrightarrow} D<\infty$ and $\Psi(x)/x \underset{x\rightarrow \infty}{\longrightarrow} -\mu<0$ (so that $-\Psi$ is the Laplace exponent of subordinator with drift $\mu$ and the Lévy process with Lévy-Khintchine function $\Sigma$ has finite variation), in which case $z^*=\frac{\mu}{D}$. \end{proposition}
We retain the piece of notation \eqref{eq:z-star}.

\subsection{Classification of attracting boundaries}\label{subsection:classification}
Recall the definition of $\mathscr{G}$ in \eqref{generatorG} and notice that $\Sigma>0$ on $(0,\infty)$ and that the local operator $\mathscr{G}$ is the generator of a certain regular diffusion on $(0,\infty)$. Let then $V:=(V_t)_{t\geq 0}$ be the minimal diffusion with generator $\mathscr{G}$, namely with boundaries $0$ and $\infty$ absorbing if they are accessible.  

 \emph{Throughout the remainder of this paper we fix an arbitrary $x_0\in (0,\infty)$}.  Then set 
\begin{equation}\label{page:S_V}
S_V(x):=\int_{x_0}^{x} \frac{1}{\Sigma(u)}e^{\int_u^{x_0}\frac{\Psi(v)}{\Sigma(v)}\ddr v}\ddr u ,\quad x\in (0,\infty),
\end{equation}
 for the scale function of $V$, see e.g. Karlin and Taylor \cite[Chapter 14, Section 6, page 227]{zbMATH03736679}. By abuse of notation denote  by $S_V$ also its associated Lebesgue-Stieltjes measure on $(0,\infty)$; to wit, for $\mathsf{a}< \mathsf{b}$ from $(0,\infty)$, \begin{equation}\label{scalemeasure} S_V(\mathsf{a},\mathsf{b}]=S_V(\mathsf{b})-S_V(\mathsf{a})=\int_\mathsf{a}^{\mathsf{b}} \frac{1}{\Sigma(x)}e^{\int_x^{x_0}\frac{\Psi(u)}{\Sigma(u)}\ddr u}\ddr x\in (0,\infty),
\end{equation}
which determines $S_V$ uniquely. The measure $S_V$ being locally finite (i.e. finite on compact subsets of $(0,\infty)$), note that if $S_V(0,\mathsf{b}]$ is infinite for some $\mathsf{b}\in (0,\infty)$ then it is so for all $\mathsf{b}\in (0,\infty)$; similarly for $S_V(\mathsf{b},\infty)$. 
Finally,  introduce
\begin{equation}\label{eq:SZ-def}
S_Z(w):=\int_{0}^{\infty}e^{-xw}S_V(\ddr x)= \int_{0}^{\infty}\frac{e^{-xw}}{\Sigma(x)}e^{\int_x^{x_0}\frac{\Psi(u)}{\Sigma(u)}\ddr u}\ddr x,\quad w\in (0,\infty).
\end{equation}

Our next theorem provides necessary and sufficient conditions for the  boundaries $0$ and $\infty$ to be attracting, by which we mean that the process tends towards the boundary with positive probability. These conditions are those of the diffusion $V$ for the boundaries $\infty$ and $0$, respectively. Recall  \eqref{eq:z-star}.
\begin{theorem}[Attracting boundaries]\label{attractiveboundaries}
Let $z^*<\mathsf{a}<z<\infty$.
\begin{enumerate}[(i)]
\item \label{attractiveboundaries:i} If $S_V(0,x_0]=\infty$ then $\mathbb{P}_z(\zeta_\mathsf{a}^-<\zeta_\infty)=1$. 
\item  \label{attractiveboundaries:ii} If $S_V(0,x_0]<\infty$ then $\mathbb{P}_z(\zeta_\mathsf{a}^-<\zeta_\infty)=\frac{S_Z(z)}{S_Z(\mathsf{a})} \in (0,1)$. 
\item \label{attractiveboundaries:iii}$Z_t\underset{t\rightarrow \infty}{\longrightarrow} \infty \text{ with positive $\PP_z$-probability}$ (respectively, $\PP_z$-almost surely) if and only if $S_V(0,x_0]<\infty$ (respectively, $S_V(0,x_0]<\infty$ and $S_V(x_0,\infty)=\infty$).
\item \label{attractiveboundaries:iv} Suppose $\Psi\ne 0$. Then $Z_t\underset{t\rightarrow \infty}{\longrightarrow} 0 \text{ with positive $\PP_z$-probability (respectively, $\PP_z$-almost surely)}$ if and only if $S_V(x_0,\infty)<\infty$ (respectively, $S_V(x_0,\infty)<\infty$ and $S_V(0,x_0]=\infty$). When $S_V(0,\infty)<\infty$,
 then, moreover, $\mathbb{P}_z(Z_t\underset{t\rightarrow \infty}{\longrightarrow} 0)=1-\mathbb{P}_z(Z_t\underset{t\rightarrow \infty}{\longrightarrow} \infty)=\frac{S_Z(z)}{S_Z(0)}\in (0,1)$. 
 \item\label{attractiveboundaries:v} If $\Psi= 0$ then $\PP_z$-almost surely $Z_t\underset{t\rightarrow \infty}{\longrightarrow} 0$.
\end{enumerate}
\end{theorem}
As we have mentioned Theorem \ref{attractiveboundaries} actually states the following correspondences:
\begin{table}[!htpb]
\begin{center}
\begin{tabular}{c|c|l}
Condition & Boundary of $Z$ &  Boundary  of $V$ \\
\hline
$S_V(0,x_0]<\infty$ & $\infty$  attracting  &  $0$ attracting \\
\hline
$\Psi=0$ or $S_V(x_0,\infty)<\infty$ & $0$  attracting  &  $\infty$ attracting\\
\hline
\end{tabular}
\vspace*{2mm}
\caption{Attracting boundaries of $Z$ and $V$}
\label{correspondance}
\end{center}
\end{table}

\vspace*{-3mm}
\begin{remark}\label{carefultransience} 
The convergence towards $\infty$ in Theorem \ref{attractiveboundaries}\eqref{attractiveboundaries:iii}, when $S_V(0,x_0]<\infty$,  hides two different possibilities: the process can either be transient ($\infty$ is attracting, but not accessible) or can explode ($\infty$ is accessible). Indeed the condition $S_V(0,x_0]=\infty$ is not necessary in general for the process to be non-explosive, see Example~\ref{exampleattracting}\eqref{exampleattracting:1} below. In the case $\Sigma(x)=\frac{c}{2}x$,  $x\in [0,\infty)$,  however, the condition $S_V(0,x_0]=\infty$  turns out to be also necessary for non-explosion  \cite[Theorem 3.1]{MR3940763}. 
No transience phenomenon can occur in logistic CBs \cite[Remark 4.9]{MR3940763}; they can only converge to $\infty$ by reaching it. 
\end{remark}

In the non-subordinator case, one can easily check that $S_V(x_0,\infty)<\infty$ always holds. So, by Theorem~\ref{attractiveboundaries}\eqref{attractiveboundaries:iv},  the necessary and sufficient condition for almost sure convergence towards $0$ is then $S_V(0,x_0]=\infty$. In the subordinator case, the condition $S_V(x_0,\infty)<\infty$ may or may  not be satisfied. In other words, collisions can be strong enough ($S_V(x_0,\infty)<\infty$) or not ($S_V(x_0,\infty)=\infty$) for the event of convergence towards $0$ to have positive probability or not. Lastly, in the (sub)critical branching case, one always has $S_V(0,x_0]=\infty$. 


\begin{example}\label{exampleattracting}  
\leavevmode
\begin{enumerate}
\item\label{exampleattracting:1}   Let $a>0$ and $b\in \mathbb{R}$. One of the simplest CBCs is the process with mechanisms \begin{center} $\Sigma(x)=\frac{a^2}{2}x^2$ and $\Psi(x)=-b x$, $x\in [0,\infty)$. \end{center} It satisfies the SDE
\[\ddr Z_t=aZ_t\ddr W_t+b Z_t\ddr t, \ Z_0=z,\]
which corresponds to a geometric Brownian motion, namely  for all $t\geq 0$,
\[Z_t=z\exp\left(\big(b-\frac{a^2}{2}\big)t+aW_t\right).\]
One can directly check that $S_V(0,x_0]<\infty$ if and only if $b>\frac{a^2}{2}$, in which case the process $(Z_t,t\geq 0)$ is transient (and does not explode). We also see that Brownian collisions regulate the deterministic growth, that is to say, $Z_t\underset{t\rightarrow \infty}{\longrightarrow} 0$ a.s., when $\frac{a^{2}}{2}>b$.
\item More generally if $\Psi'(0+)=:-b\in \mathbb{R}$ and $\Sigma(x)\underset{x\rightarrow 0}{\sim} \frac{a}{2}x^{2}$, then $S_V(0,x_0]=\infty$ if and only if $b\leq \frac{a^2}{2}$. If in addition to the latter condition $\Psi(x)> 0$ for some $x>0$, then $S_V(x_0,\infty)<\infty$ and thus $Z_t\underset{t\rightarrow \infty}{\longrightarrow} 0$ a.s.. These results are reminiscent of properties of a CB process in a Brownian environment, see Palau and Pardo \cite[Proposition 2]{zbMATH06684580}.
\item\label{exampleattracting:3} 
Consider $\Sigma(x)=dx^{\alpha}$ with $\alpha\in (1,2)$ and $\Psi(x)=-d'x^{\beta}=:-\Phi(x)$ with $\beta \in (0,1)$, $x\in [0,\infty)$.  Then we have as follows.
\begin{itemize}
\item If $\beta >\alpha-1$, neither $0$ nor $\infty$ is attracting.
\item If $\beta <\alpha-1$,  $0$ and $\infty$ are both attracting.
\item If $\beta=\alpha-1$, $\infty$ is  attracting if and only if $d'/d>\alpha-1$, while $0$ is attracting if and only if $d'/d<\alpha-1$. In the case of equality, $d'/d=\alpha-1$, neither $0$ nor $\infty$ are attracting.
\end{itemize}

\item Finally, consider the case when $\Sigma(x)=dx^{\alpha}$ for all $x\in [0,\infty)$, with $\alpha\in (1,2)$, and a branching mechanism $\Psi$ such that $\Psi'(0+)\in (-\infty,\infty)$. Then $0$ is attracting, and, if moreover $\Psi(x)\geq 0$ for some $x>0$, then $Z_t\underset{t\rightarrow \infty}{\longrightarrow} 0$ a.s..
\end{enumerate}
\end{example}

\subsection{First passage times and extinction}
We turn to the study of the law of the first passage time of the \CBC process $Z$ below a given level.  We first state a sufficient condition ensuring that $Z$ does not explode in finite time (i.e. its boundary $\infty$ is inaccessible). 
\begin{proposition}\label{propositionsuffcond} If $S_V(0,x_0]=\infty$  \text{ or } $\int_{0+}\frac{\ddr x}{-\Psi(x)}=\infty$, then 
the \CBC process does not explode.
\end{proposition}
\begin{remark}\label{H(0)finite} The fact that when $S_V(0,x_0]=\infty$ the process does not explode is a direct consequence of Theorem \ref{attractiveboundaries}\eqref{attractiveboundaries:iii}. Note that if $\int_{0}^{x_0}\frac{\Psi(u)}{\Sigma(u)}\ddr u\in (-\infty,\infty)$ then $S_V(0,x_0]=\infty$. 
The condition $\int_{0+}\frac{\ddr x}{-\Psi(x)}=\infty$ (called Dynkin's condition) is necessary and sufficient for  non-explosion of the CB$(\Psi)$, see e.g. Kyprianou \cite[Theorem 12.3]{Kyprianoubook}. In other words, and it is not surprising in view of their dynamics, collisions are never causing explosion of CBCs. 
\end{remark} 
We will find a representation of the decreasing $\theta$-invariant function of $Z$ with the help of the increasing one of the diffusion $V$. This enables us to get an expression for the Laplace transforms of the first passage times $\zeta_\mathsf{a}^-$, $\mathsf{a}\in (z^*,\infty)$. 
 
\begin{theorem}\label{firstpassagetimestheorem} 
Assume that the \CBC process does not explode and let $\theta\in (0,\infty)$. Put 
\begin{equation}\label{fthetatheorem}
f_\theta(z):=z\int_0^{\infty}e^{-zv}h_\theta(v)\ddr v,\quad z\in (z^*,\infty),
\end{equation} 
the function $h_\theta\in C^{2}((0,\infty))$ being the unique (up to a multiplicative constant\footnote{A multiplicative constant we intend always to be from $(0,\infty)$.}) nonnegative, not  zero, nondecreasing solution $h$ on $(0,\infty)$ to
\begin{equation}\label{eigenfunctionG} \mathscr{G}h=\Sigma h''+(\Sigma'+\Psi)h'=\theta   h.
\end{equation}
Then, for $\mathsf{a}\leq z$ from  $(z^*,\infty)$, 
\begin{equation}\label{LTfirstpassagetime}
\mathbb{E}_z\big[e^{-\theta \zeta^-_\mathsf{a}}\big]=\frac{f_\theta(z)}{f_\theta(\mathsf{a})}.
\end{equation}
\end{theorem}

\begin{remark} When there is no collision, $\Sigma=0$, and we are not in the subordinator case, the ordinary differential equation (o.d.e.) in \eqref{eigenfunctionG} is of  first order and there is a possible singularity at $\rho:=\sup\{x\in [0,\infty): \Psi(x)=0\}$, the largest zero of $\Psi$.  Solving the o.d.e.  gives for $v\in (\rho,\infty)$, $h_\theta(v)=e^{\int_{x_0}^{v}\frac{\theta}{\Psi(u)}\ddr u}$, where (still) $x_0\in (0,\infty)$ is fixed (and arbitrary). In turn we get   $$f_\theta(z)=z\int_\rho^{\infty}e^{-zv}e^{\int_{x_0}^{v}\frac{\theta}{\Psi(u)}\ddr u}\ddr v,\quad z\in(0,\infty),$$
and recover then through Formula \eqref{LTfirstpassagetime} the Laplace transform of the first passage time of  the CB$(\Psi)$, see \cite[Section 6, page 4192]{Duhalde}.
\end{remark}

\begin{remark}

A simple application of Tonelli's theorem ensures that the so-called scale function $f_\theta$ in \eqref{fthetatheorem} satisfies $f_\theta(z)=h_\theta(0+)+\int_0^\infty e^{-zv}h'_\theta(v)\dd v$, $z \in (0,\infty)$. In particular, $f_\theta$ is completely monotone. This phenomenon of  ``complete monotonicity at first passage'' was recently noted and explored for time-changed spectrally positive L\'evy processes in Vidmar \cite{vidmar-cm}. 
\end{remark}
Theorem~\ref{firstpassagetimestheorem} deals with first passage times below (accessible) positive levels. As for the first passage time to zero of $Z$ and the event of extinction we offer
\begin{theorem}\label{LTextinctiontheorem} Assume that the \CBC process does not explode. Let $z\in (z^*,\infty)$. The following equivalence holds true.
\begin{equation}\label{equiv} \mathbb{P}_z(\zeta_0^-<\infty)>0 \text{ if and only if }\Psi(\infty)=\infty\text{ and } \int^{\infty}\frac{\ddr u}{\Psi(u)}<\infty.
\end{equation}
Furthermore,  the Laplace transform of the extinction time of $Z$ satisfies: \begin{equation}\label{idLT}\mathbb{E}_z\big[e^{-\theta \zeta^-_0}\big]=\int_{0}^{\infty}ze^{-zx}\frac{h_\theta(x)}{h_\theta(\infty)}\ddr x=\mathbb{E}\big[ e^{-\theta \tau_\infty^{\mathbbm{e}_z}}\big],\quad \theta\in (0,\infty),\end{equation}
where $\tau_\infty^{\mathbbm{e}_z}$ denotes the explosion time of $V$ when the latter starts from an independent exponential random variable $\mathbbm{e}_z$ with rate $z$. \end{theorem}
\begin{remark} The integrability condition $\Psi(\infty)=\infty$, $\int^{\infty}\frac{\ddr u}{\Psi(u)}<\infty$ (called Grey's condition) is necessary and sufficient for the CB$(\Psi)$ process to become extinct with positive probability, see e.g. \cite[Theorem 3.1.3]{Li-book}. Collisions are therefore also never causing extinction in finite time of a non-explosive population. 
\end{remark}
\begin{remark}
Identity~\eqref{idLT} reveals that under the assumption of non-explosion of $Z$, the boundary $0$ is accessible for $Z$ if and only if $\infty$ is accessible for $V$. It was established via different arguments for logistic CBs and their extensions in \cite[Theorem 3.2]{foucart2021local}.
\end{remark}


\subsection{Stationary distribution}  As observed in Example~\ref{exampleattracting}\eqref{exampleattracting:3}, in the subordinator case, some phenomenon of recurrence can occur and a stationary regime may exist. Let $M_V$ be the speed measure of $V$ on $(0,\infty)$: for $\mathsf{a}<\mathsf{b}$ from $(0,\infty)$, \begin{equation}\label{speedmeasure}
M_V(\mathsf{a},\mathsf{b}]=\int_\mathsf{a}^{\mathsf{b}}e^{\int_{x_{0}}^x\frac{\Psi(u)}{\Sigma(u)} \ddr u}\ddr x\in (0,\infty),
\end{equation}
where, still, $x_0\in (0,\infty)$ is arbitrary but fixed.

\begin{theorem}[Stationary distribution and long-term behavior]\label{stationarydisttheorem}
 Assume that $S_V(0,x_0]=\infty$ and $S_V(x_0,\infty)=\infty$. Let $z\in (0,\infty)$. Then  the minimal CBC process converges in law towards a non-degenerate random variable $Z_\infty$ on $(z^*,\infty)$ if and only if  $M_V(0,\infty)<\infty$. Moreover, the Laplace transform of the latter is then given by \begin{equation}\label{LTstationary}\mathbb{E}_z[e^{-xZ_\infty}]=\frac{M_V(x,\infty)}{M_V(0,\infty)},\quad  x\in[ 0,\infty).\end{equation}
The case $M_V(0,\infty)=\infty$ covers three different possibilities:
\begin{enumerate}[(i)]
\item If $M_V(0,x_0]<\infty$ and $M_V(x_0,\infty)=\infty$, then $Z_t\underset{t\rightarrow \infty}{\longrightarrow} 0$ in probability.
\item If $M_V(0,x_0]=\infty$ and $M_V(x_0,\infty)<\infty$, then $Z_t\underset{t\rightarrow \infty}{\longrightarrow} \infty$ in probability.
\item If $M_V(0,x_0]=\infty$ and $M_V(x_0,\infty)=\infty$, then 
$Z$ has no limiting distribution. 
\end{enumerate}

\end{theorem}

\begin{remark} Plainly, if $-\Psi$ is not the Laplace exponent of a subordinator, then $M_V(x_0,\infty)=\infty$. Here are two simple conditions ensuring, between them, that $M_V(0,\infty)<\infty$ and hence that a limiting distribution exists. If $\Sigma'(0+)=c/2>0$, then $M_V(0,x_0]<\infty$ (without further assumptions on $\Psi$). 
If $-\Psi$ is the Laplace exponent of a subordinator with drift $\mathrm{d}$, i.e. $-\Psi(x)/x \underset{x\rightarrow \infty}{\longrightarrow} \mathrm{d}>0$, such that $\frac{2\mathrm{d}}{a^2}>1$ (with $a\geq 0$ the diffusive coefficient in \eqref{collisionmechanism} and by convention $1/0=\infty$), then $M_V(x_0,\infty)<\infty$.
\end{remark}
\begin{remark} One verifies easily from \eqref{LTstationary} that the limiting distribution of the \CBC admits a first moment if and only if $\int_0^{x_0}\frac{-\Psi(u)}{\Sigma(u)}\ddr u<\infty$.
\end{remark}
\begin{example} \label{example:legion} \leavevmode
\begin{enumerate}
\item\label{example:legion-verhulst} Consider the \CBC process with collisions and branching mechanisms satisfying,  for  $x\in[ 0,\infty)$, $\Sigma(x)=\frac{a^2}{2}x^2+\frac{c}{2}x$ with $a,c\in (0,\infty)$ and $\Psi(x)=-\mu x$ with $\mu\in \mathbb{R}$. In other words, $(Z_t,t\geq 0)$ satisfies the SDE, called stochastic Verhulst equation
\[\ddr Z_t=aZ_t\ddr W_t+\big(\mu Z_t-\frac{c}{2}Z_t^{2}\big)\ddr t,\quad  Z_0=z.\]
See Giet et al. \cite{zbMATH06630288} for a deep study of this diffusion (including its first passage times). The \CBC process $Z$ admits a limiting distribution if and only if  $\mu>\frac{a^2}{2}$. When it exists, the latter has for its Laplace transform \[\mathbb{E}[e^{-xZ_\infty}]=\left(\frac{a^2}{c}x+1\right)^{-(\frac{2\mu}{a^2}-1)},\quad x\in [0,\infty),\]
which is the Laplace transform of a gamma distribution with density \[(0,\infty)\ni u\mapsto \frac{\beta^{\alpha}}{\Gamma(\alpha)}u^{\alpha-1}e^{-\beta u},\] 
its parameters being $\alpha:=\frac{2\mu}{a^2}-1$ and $\beta:=\frac{c}{a^2}$.
\item Assume that, for $x\in [0,\infty)$, $\Sigma(x)=dx^{\alpha}$ with $\alpha\in (1,2)$ and $\Psi(x)=-d'x^{\beta}$ with $\beta >\alpha-1$ and $d,d'\in (0,\infty)$. Then the \CBC process $Z$ admits a limiting distribution with Laplace transform:
\[\mathbb{E}[e^{-xZ_\infty}]=\frac{\int_x^{\infty}e^{-\frac{d'}{d}u^{\beta-\alpha+1}}\ddr u}{\int_0^{\infty}e^{-\frac{d'}{d}u^{\beta-\alpha+1}}\ddr u}=\frac{\Gamma\left(\frac{1}{\beta-\alpha+1},\frac{d'}{d}x^{\beta-\alpha+1}\right)}{\Gamma\left(\frac{1}{\beta-\alpha+1}\right)}, \quad  x\in [0,\infty),\]
where $\Gamma(s,x):=\int_{x}^{\infty}u^{s-1}e^{-u}\ddr u$ (for $s>0$ and $x\geq 0$) is the incomplete Gamma function.
\item Assume that, for $x\in [0,\infty)$, $\Sigma(x)=dx^{\alpha}$ with $\alpha\in (1,2)$, $d\in (0,\infty)$ and $\Psi(x)=-dx^{\alpha-1}$. Then if $d'/d<1$, $M_V(x_0,\infty)=\infty$ and $M_V(0,x_0]<\infty$, thus $Z$ tends to $0$ in probability. If $d'/d>1$, $M_V(x_0,\infty)<\infty$ and $M_V(0,x_0]=\infty$, and again $Z$ tends to $\infty$ in probability. In the case $d'/d=1$, $Z$ has no limiting distribution. 
\end{enumerate}
\end{example}

\subsection{The role of Laplace and Siegmund dualities}The second order differential operator $\mathscr{G}$, defined in \eqref{generatorG}, will first appear as an analytical trick  in the quest for an eigenfunction of $\mathscr{L}$, see the proof of Theorem \ref{firstpassagetimestheorem}, especially the forthcoming Lemma \ref{proposition:h-to-f}. The link between the generators $\mathscr{G}$, $\mathscr{A}$ and $\mathscr{L}$ hinges in fact on two duality relationships, known as \textit{Laplace duality} and \textit{Siegmund duality}. We explore now these dualities, which will for instance allow us to represent, under certain conditions, the semigroup of the process $Z$ with that of $U$ and in turn $V$.

From \eqref{generatorA} and \eqref{genZ} one checks  by direct computation the  key identity
\begin{equation}\label{dualgen}
\mc{L}_ze^{-xz}=\Sigma(x)z^2e^{-xz}+\Psi(x)ze^{-xz}=\mc{A}_xe^{-xz},\quad \{x,z\}\subset [0,\infty).
\end{equation}
We say that Laplace duality \eqref{dualgen} holds  \textit{at the level of the generators}. Under the assumption of non-explosion of $Z$ we have moreover the following duality relationship  \textit{at the level of the semigroups}.
\begin{proposition}\label{Laplacedualsemigroup} Let $Z$ be the  \CBC and $U$ the diffusion with generator $\mathscr{A}$ and $0$ an absorbing state. Assume that $Z$ does not explode. Then 
\begin{equation}\label{dualsemigroup}
\mathbb{E}_z[e^{-xZ_t}]=\mathbb{E}_x[e^{-zU_t}],\quad \{t,x,z\}\subset [ 0,\infty).
\end{equation}
\end{proposition}
\begin{remark} \label{remark:feller}
Proposition \ref{Laplacedualsemigroup} requires the non-explosiveness of the process $Z$; this assumption will play an important role in the proof. On the other hand, the diffusion $U$ is automatically non-explosive, as we shall prove in due course (Lemma~\ref{lemma:U-not-explocive}). The duality allows one to represent the semigroup of $Z$ with the help of that of $U$. In particular, one can check from \eqref{dualsemigroup} that under non-explosion the semigroup of the CBC process $Z$ is Feller, see \cite[Lemma 6.3]{MR3940763}. 
\end{remark}
Under extra conditions, which guarantee  that $V$ has no attracting boundaries, the diffusion $U$ in turn is in Siegmund duality with the diffusion $V$, in the following precise sense.

\begin{proposition}\label{Siegmundualsemigroup} Assume that $S_V(0,x_0]=\infty$, $S_V(x_0,\infty)=\infty$, and recall $U$ is the  diffusion with generator $\mathscr{A}$.  For all $x,y\in (0,\infty)$, 
\begin{equation}\label{Siegmunddual} \mathbb{P}_x(U_t<y)=\mathbb{P}_y(V_t>x),
\end{equation}
where $V$ is the diffusion with generator $\mathscr{G}$. Moreover, for any $z\in (0,\infty)$, one has
$$\mathbb{E}_z[e^{-xZ_t}]=\int_{0}^{\infty}ze^{-zy}\mathbb{P}_y(V_t>x)\ddr y,\quad x\in [0,\infty).$$
\end{proposition}

The final substantial result on which we report here establishes that, in a sense that shall be be made precise presently, Laplace duality with a diffusion at the level of the generators \eqref{dualgen} actually characterizes CBCs. In order to formulate this with ease we suspend temporarily all meaning attached hitherto to $Z$, $\mathscr{L}$, $\mathscr{A}$, $\Sigma$ and $\Psi$ (and indeed just all the notation introduced thusfar). 
\begin{theorem}\label{thm:laplace-converse}
Let $\mathscr{L}$ be the infinitesimal generator of a positive (possibly explosive) Feller process $(Z_t,t\geq 0)$ without negative jumps and $0$ an absorbing state, whose domain includes\footnote{Rapidly decaying functions are those $f\in C^{\infty}([0,\infty))$ (admitting a $C^\infty$ extension to a neighborhood of $[0,\infty)$) such that $\underset{z\rightarrow \infty}{\lim} P(z)f^{(k)}(z)=0$ for any polynomial $P$ and any $k\in \mathbb{N}_0$.}  
\begin{align*}
\mathcal{S}:=\{f\in \mathbb{R}^{[0,\infty)}:&\,(\text{the limit } f(\infty):=\lim_{u\to\infty} f(u)\text{ exists in  }\mathbb{R})\text{ and }\\
&\,(f-f(\infty)\in \text{ Schwartz space of rapidly decaying functions})\},
\end{align*}
mapping $\mathcal{S}$ into $C_0([0,\infty))$. Suppose  further $\mathscr{L}$ is in Laplace duality with the conservative generator of a diffusion process on $[0,\infty)$, more precisely, suppose that
 \begin{equation}\label{eq.duality-converse}
\mc{L}_ze^{-xz}=\Sigma(x)z^2e^{-xz}+\Psi(x)ze^{-xz}=:\mathscr{A}_xe^{-xz},\quad \{x,z\}\subset [0,\infty),
\end{equation}
 holds true for some $\Sigma:[0,\infty)\to [0,\infty)$, not zero, and some $\Psi:[0,\infty)\to \mathbb{R}$, both continuous at zero. Then $\Psi$ and $\Sigma$ are  L\'evy-Khintchine functions of the spectrally positive type as in \eqref{branchingmechanism}-\eqref{collisionmechanism} and $\mathscr{L}$ acts  on $C_ c^\infty([0,\infty))$ according to  \eqref{genZ}.
\end{theorem}
We shall see later in Corollary~\ref{corollaray:genZ} that CBCs actually meet the property assumed on the generator of $Z$ above, so, together with the Feller property  noted in Remark~\ref{remark:feller}, this really is a characterization of  non-explosive CBCs through Laplace duality with diffusions.  Remark also that in the general theory of Feller processes the infinitesimal generator is usually only defined on continuous maps vanishing at infinity. For the complete formulation of the Laplace duality it is however convenient to include in the domain the constants, hence our slight departure from this convention. To avoid any ambiguity, for a measurable $f:[0,\infty)\to \mathbb{R}$, with a limit (in $\mathbb{R}$) at $\infty$ [when $Z$ can explode], by  ``$f$ being in the domain of the infinitesimal generator $\mathscr{L}$ of $Z$'' we mean simply that, for all $z\in [0,\infty)$, $\mathscr{L}f(z)=\lim_{t\downarrow 0}\frac{\EE_z[f(Z_t)]-f(z)}{t}$ (with the interpretation $f(\infty):=\lim_{u\to\infty} f(u)$), the limits existing in $\mathbb{R}$. Consistently with these conventions, $\mathscr{L}$ always includes in its domain and annihilates the constants. In particular, in \eqref{eq.duality-converse}, $\mathscr{L}1=0$ (when $x=0$).
\section{Construction and Lamperti representation of CBCs}\label{ProofofTheorem1}
\subsection{Study of stochastic equation \eqref{cbcequation}: proof of Theorem~\ref{theorem:cbc-sde-contruction}}
Stochastic equations of the form \eqref{cbcequation} fall into the general class of certain SDEs with jumps studied by Dawson and Li \cite{DawsonLi}, Fu and Li \cite{FU2010306}, and Palau and Pardo \cite{zbMATH06836271}. Here we are able to apply directly the result \cite[Proposition~1]{zbMATH06836271}, recognizing that \eqref{cbcequation} is just \cite[Eq.~(5)]{zbMATH06836271} 
\begin{align*}
Z_t=&z+\int_0^t\underline{b}(Z_s)\dd s+\int_0^t\int_{\underline{E}}\underline{\sigma}(Z_s,u) \underline{W}(\dd s,\dd u)\\
&\quad +\int_0^t\int_{\underline{U}}\underline{g}(Z_{s-},u)\underline{M}(\dd s,\dd u)+\int_0^t\int_{\underline{V}} \underline{h}(Z_{s-},v)\overline{ \underline{N}}(\dd s,\dd v),
\end{align*}
 with the following input data of \cite{zbMATH06836271} in which we underline the objects of \cite{zbMATH06836271} to avoid confusion with our own: 
\begin{align*}
\underline{b}(z)&=bz-\frac{c}{2}z^2,\quad z\in [0,\infty)\\
\underline{E}&=\{1,2\}\\
\underline{\sigma}(z,1)&=\sigma\sqrt{z},\quad z\in [0,\infty)\\
\underline{\sigma}(z,2)&=az,\quad z\in [0,\infty)\\
\underline{W}(\dd s,\dd e)&=B(\dd s)\delta_1(\dd e)+W(\dd s)\delta_2(\dd e)\\
&\qquad \text{[white noise on $(0,\infty)\times E$ with intensity $\dd s\underline{\pi}(\dd e)$]}\\
\underline{\pi}&=\delta_1+\delta_2\\
\underline{U}&=[0,\infty)\times (1,\infty)\\
\underline{g}(z,(u,h))&=h\mathbbm{1}_{(0,z]}(u),\quad (z,(u,h))\in [0,\infty)\times \underline{U}\\
\underline{M}(\dd s,\dd u,\dd h)&=\mathcal{N}\vert_{[0,\infty)\times ([0,\infty)\times (1,\infty))}(\dd s,\dd u,\dd h) \\
&\qquad \text{ [Poisson random measure with intensity $\dd s\underline{\mu}(\dd (u,h))$]}\\
\underline{\mu}(\dd (u,h))&=\dd u\pi(\dd h)\\
\underline{V}&=([0,\infty)\times (0,1])\cup ([0,\infty)\times [0,\infty)\times (0,\infty))\\
\underline{h}(z,(u,h))&=h\mathbbm{1}_{(0,z]}(u),\quad (z,(u,h))\in [0,\infty)\times ([0,\infty)\times (0,1])\\
\underline{h}(z,(u_1,u_2,h))&=h\mathbbm{1}_{(0,z]\times (0,z]}(u_1,u_2),\quad (z,(u_1,u_2,h))\in [0,\infty)\times ([0,\infty)\times [0,\infty)\times (0,\infty))\\
\underline{N}(\dd s,\dd u,\dd h)&=\mathcal{N}(\dd s,\dd u,\dd h)\text{ on }[0,\infty)\times ([0,\infty)\times  (0,1])\\
\underline{N}(\dd s,\dd u_1,\dd u_2,\dd h)&=\mathcal{M}(\dd s,\dd u_1,\dd u_2,\dd h)\text{ on } [0,\infty)\times ([0,\infty)\times [0,\infty)\times (0,\infty))\\
\underline{\nu}(\dd (u,h))&=\dd u\pi(\dd h) \text{ and } \underline{\nu}(\dd (u_1,u_2,h))=\dd u_1\dd u_2\eta(\dd h)\\
&\qquad \text{[characteristic measure of $\underline{N}$]}.
\end{align*}
Then the admissibility conditions (i)-(iv) of \cite[p.~60]{zbMATH06836271} are met evidently: (i) $b$ is continuous nonnegative; (ii) $\sigma$ is continuous and vanishing at zero in the first entry; (iii) $g$ is Borel and majorizing minus the identity in the first entry; (iv) $h$ is Borel, vanishing at zero and majorizing minus the identity in the first entry. Choosing $\underline{\tilde{U}}=\underline{U}$ we have $\underline{\mu}(\underline{U}\backslash \underline{\tilde U})=0$ trivially but also $$\int_{\underline{\tilde U}} \vert \underline{g}(z,(u,h))\vert \land 1\underline{\mu}(\dd (u,h))\leq z\pi(1,\infty),\quad z\in [0,\infty),$$
which verifies [is]  (a) of \cite[p.~60]{zbMATH06836271}. Choosing $\underline{b_1}(z)=bz$ and $\underline{b_2}(z)=\frac{c}{2}z^2$ and putting $\underline{r_n}(z)=[1\lor (\vert b\vert+\int n\land h\pi(\dd h))]z$ for $n\in \mathbb{N}_0$ we get  (b) of \cite[p.~60]{zbMATH06836271}: $\underline{b}=\underline{b_1}-\underline{b_2}$, $\underline{b_1}$ continuous, $\underline{b_2}$ nondecreasing;  $\underline{r_n}$  nondecreasing concave, $\int_{0+}\underline{r_n}^{-1}=\infty$ and (since, for $u\in [0,\infty)$, $\vert (\mathbbm{1}_{(0,x]}(u)h)\land n-(\mathbbm{1}_{(0,y]}(u)h)\land n\vert=(h\land n) \mathbbm{1}_{(0,x]\triangle (0,y]}(u)$) $$\vert \underline{b_1}(x)-\underline{b_2}(y)\vert+\int_{\underline{\tilde{U}}}\vert \underline{g}(x,(u,h))\land n-\underline{g}(y,(u,h))\land n\vert\underline{\mu}(\dd (u,h))\leq \underline{r_n}(\vert x-y\vert)$$ for all $n\in \mathbb{N}_0$ and $\{x,y\}\subset [0,\infty)$. One also easily verifies (c) of \cite[p.~61]{zbMATH06836271} taking into account that $\int h\land h^2\eta(\dd h)<\infty$ and $\int_{(0,1]} h^2\pi(\dd h)<\infty$: $z\mapsto \underline{h}(z,v)+z$ is nondecreasing; and,  for each $n\in \mathbb{N}_0$, there is a $\underline{B_n}<\infty$ such that for $\{x,y\}\subset [0,n]$ we have
$$\int_{\underline{E}}\vert \underline{\sigma}(x,u)-\underline{\sigma}(y,u)\vert\underline{\pi}(\dd u)+\int_{\underline{V}}\vert \underline{l}(x,y,v)\vert\land \underline{l}(x,y,v)^2\underline{\nu}(\dd v)\leq \underline{B_n}\vert x-y\vert$$ with $\underline{l}(x,y,v):=\underline{h}(x,v)-\underline{h}(y,v)$ for $v\in \underline{V}$.

All in all the preceding allows us to infer the conclusion of \cite[Proposition~1]{zbMATH06836271}, which is, that for each starting value $z\in [0,\infty)$ there is an a.s. unique $[0,\infty]$-valued c\`adl\`ag process $Z$, adapted to the natural filtration of $(\underline{W},\underline{M},\underline{N})$, that is to say, of $(W,B,\mathcal{N},\mathcal{M})$, such that \eqref{theorem:cbc-sde-contruction:ii}-\eqref{theorem:cbc-sde-contruction:i} of Theorem~\ref{theorem:cbc-sde-contruction} hold true. 

It is clear from \eqref{cbcequation} that $Z$ a.s. has no negative jumps and  that $0$ is an absorbing state for $Z$. Quasi left-continuity also follows directly from \eqref{cbcequation} because the jump times of a homogenous Poisson process are not announcable, while the integrals against the Brownian motions and the Lebesgue integrals are anyway continuous. The proof of the strong Markov property is essentially the same as for the CBM processes \cite[Theorem~2.1(iii)]{vidmar2021continuousstate} and boils down to the strong Markov property for the Brownian and Poisson drivers of \eqref{cbcequation}; we omit the details. Finally, by It\^o's formula, see e.g. Ikeda and Watanabe \cite[Theorem~II.5.1]{ikeda1989stochastic}, and \eqref{cbcequation} again, the martingale conclusion of Theorem~\ref{theorem:cbc-sde-contruction} follows (to see how such a computation evolves on a technical level the reader may again consult the CBM case \cite[Theorem~2.1(v)]{vidmar2021continuousstate}, there is no fundamental difference).

The fact that the law of $Z$ is uniquely determined by the triplet $(\Sigma,\Psi,z)$ follows from the observation that pathwise uniqueness implies uniqueness in law for SDEs, which completes the proof of Theorem~\ref{theorem:cbc-sde-contruction}.\qed

Let us justify (further) our calling $\mathscr{L}$ the generator of $Z$. We use, in what follows, the conventions concerning the infinitesimal generators to be found immediately after the statement of  Theorem~\ref{thm:laplace-converse}.

\begin{corollary}\label{corollaray:genZ}
Suppose (a) $f\in C^2([0,\infty))$  has a finite limit at infinity and (b)  $\mathscr{L}f$ is  vanishing at infinity.  Then the process $(f(Z_t)-\int_0^t \mathscr{L}f(Z_s)\dd s,t\geq 0)$ is a  martingale and $\mathscr{L}f$  gives the action of the infinitesimal generator of $Z$ on $f$ (here we understand $f(\infty):=\lim_{u\to\infty} f(u)$ and $\mathscr{L}f(\infty):=0$, of course). Any function from  the set $$\mathcal{D}:=\left\{f\in C^2([0,\infty)): \exists\,\underset{z\rightarrow \infty}{\lim}  f(z)\in \mathbb{R}  \text{ and } \underset{z\rightarrow \infty}{\lim} z^2\big(\vert f(z)-f(\infty)\vert+\vert f'(z)\vert+\vert f''(z)\vert \big)=0\right\}$$ meets the properties (a)-(b). 
\end{corollary}
\begin{example}\label{example:in-domain-genZ}
The  space $\mathcal{S}$ of Theorem~\ref{thm:laplace-converse}, a fortiori $C^\infty_c([0,\infty))$, is a subset of $\mathcal{D}$. In particular, for $x \in [0,\infty)$ the exponential map $([0,\infty)\ni z\mapsto e^{-xz})$  belongs to $\mathcal{D}$; moreover, $\mathcal{C}:=\{\int e^{-\cdot x}\nu(\dd x):\nu\text{ a finite signed measure on }\mathcal{B}_{[0,\infty)}\}\subset \mathcal{D}$. 
\end{example}
\begin{remark}\label{remark:action}
With a view towards the Laplace duality of \eqref{dualgen} it is perhaps worth noting explicitly that for  $f\in C^2_b([0,\infty))$ with $\mathscr{A}f$ bounded, for the exponential maps in particular, likewise the process  $(f(U_t)-\int_0^t \mathscr{A}f(U_s)\dd s,t\geq 0)$ is a martingale and $\mathscr{A}f$  gives the action of the infinitesimal generator of $U$ on $f$.
\end{remark}
\begin{proof}[Proof of Corollary~\ref{corollaray:genZ}]
Note that $f$ and $\mathscr{L}f$ are both bounded and continuous. Let $S$ be a bounded stopping time. The local martingale of \eqref{cbc-construction:i:a} with $\alpha=0$ is bounded up to every bounded time, therefore a martingale. Sampling this martingale at $S$ we get $$\EE_z\left[ f(Z_{S\land \zeta_n^+})-\int_0^{S\land \zeta_n^+} \mathscr{L}f(Z_s)\dd s\right]=f(z),\quad n\in [0,\infty).$$ Letting $n\to\infty$, $ f(Z_{S\land \zeta_n^+})\to f(Z_{S\land \zeta_\infty})$ boundedly (since $f(\infty)\in \mathbb{R}$) and $\int_0^{S\land \zeta_n^+} \mathscr{L}f(Z_s)\dd s\to \int_0^{S\land \zeta_\infty} \mathscr{L}f(Z_s)\dd s$ boundedly by bounded convergence (for the Lebesgue integral). By bounded convergence in the displayed formula we infer that $$\EE_z\left[ f(Z_{S\land \zeta_\infty})-\int_0^{S\land \zeta_\infty} \mathscr{L}f(Z_s)\dd s\right]=f(z).$$ Since $Z=\infty$ on $[\zeta_\infty,\infty)$ and since $\mathscr{L}f(\infty)=0$ we may get rid of ``$\land \zeta_\infty$''. It being true for arbitrary $S$ entails that $(f(Z_t)-\int_0^t \mathscr{L}f(Z_s)\dd s,t\geq 0)$ is a martingale that is even bounded up to every bounded time. The second claim now follows easily: 
$$\lim_{t\downarrow 0}\frac{\EE_z[f(Z_t)]-f(z)}{t}=\lim_{t\downarrow 0}\frac{\EE_z[\int_0^t\mathscr{L}f(Z_s)\dd s]}{t}=\mathscr{L}f(z),$$
by bounded convergence and the continuity of $\mathscr{L}f$ (and the right-continuity of $Z$ at time zero).

 Take now  $f\in \mathcal{D}$ and we check that $\underset{z\rightarrow \infty}{\lim}\mathscr{L}f(z)=0$. Since $\mathscr{L}$ annihilates the constants we may and do assume that $\lim_\infty f=0$. Since $f\in C^2([0,\infty))$ one has for it the following Taylor formula with integral form of  remainder, see e.g. Zorich \cite[page 363]{zbMATH02008480}, 
$$f(z+h)-f(z)-hf'(z)=h^2\int_{0}^{1}f''(z+hv)(1-v)\ddr v,\quad \{z,h\}\subset [0,\infty).$$
Recalling \eqref{genLevy}-\eqref{genZ} we estimate
\begin{align*}
z\left \lvert \mathrm{L}^{\Psi}f(z)\right \lvert&\leq \frac{\sigma^2}{2}\cdot z\vert f''(z)\vert +\vert b\vert\cdot z\vert f'(z)\vert\\
&\quad + z\left \lvert \int_0^{1}\pi(\ddr h)h^2\int_{0}^{1}f''(z+hv)(1-v)\ddr v \right\lvert+ z\left \lvert \int_{1}^{\infty}\pi(\ddr h)\left(f(z+h)-f(z)\right)\right \lvert\\
&\leq \frac{\sigma^2}{2}\cdot z\vert f''(z)\vert +\vert b\vert\cdot z\vert f'(z)\vert+\int_{0}^{1}\pi(\ddr h)h^2\int_{0}^{1}(z+hv)\lvert f''(z+hv) \lvert (1-v)\ddr v\\
&\quad +\int_1^\infty (z+h)\lvert f(z+h)\lvert \pi(\ddr h)+ \pi(1,\infty) z\lvert f(z)\lvert,\quad z\in [0,\infty).
\end{align*}
Now, the terms $(z+hv)f''(z+hv)$ and $(z+h)f(z+h)$ converge towards $0$ as $z$ goes to $\infty$ uniformly for positive $h$, $v$, and in particular are bounded. Hence by dominated convergence, both integrals in the upper bound above vanish when $z$ goes to $\infty$. The same is true for the other terms. Similar calculations entail that $z^2\left \lvert \mathrm{L}^{\Sigma}f(z)\right \lvert$ converges to $0$ as $z$ goes to $\infty$, which then allows us to conclude.
\end{proof}

\subsection{CBCs as time-changes of CBMs}
CBMs are a kind-of generalization of CBIs in which, roughly speaking, the immigration subordinator is replaced by a spectrally positive Lévy process (this will be the process $X$ in \eqref{CBMSDE} below, $Y$ being the CBM). Though, CBMs are stopped when reaching $0$, while CBIs are not. The Brownian part, and drift when it is negative, of $X$ are interpreted as migration (emigration/immigration) in the population. Such processes were defined and studied by Vidmar in \cite{vidmar2021continuousstate}. 

CBCs may be connected to  CBMs via time-change. On an heuristic level this is clear from the form of their generators. Indeed, comparing \eqref{genZ} with \cite[Eq.~(2.1)]{vidmar2021continuousstate}, which gives the action of the generator $\mathscr{L}'$ of a CBM $Y$  with branching mechanism $\Sigma$ and migration mechanism $\Psi$ on a $C^2_b([0,\infty))$ map $g$ satisfying $\mathrm{L}^\Psi g(0)=0$ as $$\mathscr{L}'g(y):=\mathrm{L}^\Psi g(y)+y\mathrm{L}^\Sigma g(y),\quad y\in [0,\infty),$$
strongly suggests that $Z$ should be just a Lamperti-type transform of a such  a CBM  by the inverse of $\int_0^\cdot\frac{\dd u}{Y_u}$. It is indeed so:

\begin{theorem}\label{thm:time-change}
Put $\kappa:=\int_0^{\zeta_\infty}Z_t\dd t$, define the additive functional $\gamma:=\int_0^\cdot Z_t\dd t$ and let  $\gamma^{-1}$ be its inverse on $[0,\kappa)$, extended by $\zeta_0^-\land\zeta_\infty$ on $[\kappa,\infty)$. Set $Y:=Z_{\gamma^{-1}}$, defined on $[0,\zeta)$, $\zeta:=\int_0^{\zeta_\infty}Z_u\dd u+\infty\mathbbm{1}_{\{\zeta^-_0<\zeta_\infty\}}$. Then $Y=(Y_u)_{u\in[0,\zeta)}$ is a CBM process with branching mechanism $\Sigma$, migration mechanism $\Psi$ (and initial value $z$), $\zeta=\infty$ a.s. (non-explosivity) and letting $\omega$ be the right-continuous inverse of $\int_0^\cdot\frac{\dd u}{Y_u}$ on $[0,\int_0^{\zeta}\frac{\dd u}{Y_u})$ [with the understanding $1/0=\infty$] we have a.s. $$\zeta_\infty=\int_0^{\infty}\frac{\dd u}{Y_u}\quad \text{ and }\quad Z_t=Y_{\omega(t)}\text{ for } t\in [0,\zeta_\infty).$$ 
\end{theorem}
Thus the CBC $Z$ may be viewed as the process got by driving along the sample paths of the (non-explosive) CBM $Y$ with a velocity that is given by its position.

%
\begin{proof}

We time-change  \eqref{cbcequation} into an SDE for the process $Y$.

By definition $\gamma^{-1}$ is a continuous time-change for the filtration $\FF$. Possibly by enlarging the underlying probability space we grant ourselves access to the following mutually independent stochastic items, independent also of $(B,W,\mathcal{M},\mathcal{N})$: Brownian motions $\tilde B$, $\tilde W$; Poisson random measures $\tilde{\mathcal{M}}(\dd s,\dd u,\dd h)$ with intensity  $\dd s\dd u\eta(\dd h)$,  $\tilde{\mathcal{N}}(\dd s,\dd h)$ with intensity  $\dd s\pi(\dd h)$. Let $\FF'$ be $\FF_{\gamma^{-1}}$ enlarged by the natural filtration of $(\tilde B,\tilde W,\tilde M,\tilde N)$ and augmented.

 Put $B':=\int_0^{\gamma^{-1}(\cdot)}\sqrt{Z_s}\dd B_s$ on $[0,\kappa)$ and extend it by the increments of $\tilde B$ after $\kappa$. By the martingale characterization of Brownian motion \cite[Theorem~II.6.1]{ikeda1989stochastic} it follows that $B'$ is an $\FF'$-Brownian motion. In the same manner we procure an $\FF'$-Brownian motion $W'$. The covariation process of $W'$ with $B'$ vanishes; thus $W'$ and $B'$ are actually independent $\FF'$-Brownian motions. 
 
 Next define $\mathcal{M}'([0,t]\times L\times A):=\int_0^{\gamma^{-1}(t)}\int_{0}^{Z_{s-}}\int_L\int_A\mathcal{M}(\dd s,\dd u_1,\dd u_2,\dd h)$ for $t\in [0,\kappa)$ and Borel $L$, $A$. The measure $\mathcal{M}'$ is extended in the first coordinate from $[0,\kappa)$ to $[0,\infty)$ by using $\tilde{\mathcal{M}}$ on $[\kappa,\infty)$. From the martingale characterization of Poisson point processes \cite[Theorem~II.6.2]{ikeda1989stochastic} it follows that $\mathcal{M}'(\dd s,\dd u,\dd h)$ is an $\FF'$-Poisson random measure with intensity measure $\dd s\dd u\eta(\dd h)$. In an analogous way we avail ourselves of an $\FF'$-Poisson random measure $\mathcal{N}'(\dd s,\dd h)$ of intensity $\dd s\pi(\dd h)$.  The  Poisson point processes (corresponding to) $\mathcal{M}'$ and $\mathcal{N}'$ a.s. have no jumps in common; therefore are actually independent.
 
Being defined in the common filtration $\FF'$, the  Brownian pair $(W',B')$ and Poisson pair $(\mathcal{N}',\mathcal{M}')$ are also automatically independent. Thus $W',B',\mathcal{N}',\mathcal{M}'$ are jointly independent.

Rewriting  \eqref{cbcequation} in terms of $Y$ we get a.s.
\begin{equation}\label{CBMSDE} Y_t= X_{t\land \sigma_0}+a\int_{0}^{t}\!\!\sqrt{Y_s} \ddr W'_s-\frac{c}{2}\int_{0}^{t}Y_s\ddr s+\int_{0}^{t} \int_0^{Y_{s-}}\int_0^\infty h\bar{\mathcal{M}}'(\ddr s,\ddr u,\ddr h), \quad t\in [0,\zeta),
\end{equation}
where 
$$X_t:=z+\sigma B'_t+bt  +\int_{0}^{t}\int_0^{1}  h\bar{\mathcal{N}'}(\ddr s,\ddr h)+\int_{0}^{t}\int_1^\infty  h\mathcal{N}'(\ddr s,\ddr h),\quad t\in [0,\infty),$$ and where $\sigma_0:=\inf\{t\in [0,\zeta):Y_t=0\}$; also $\sup_{[0,\zeta)}Y=\infty$ a.s. on $\{\zeta<\infty\}$, by construction. It follows from \cite[Theorem~2.1]{vidmar2021continuousstate} that $Y$ is a CBM with branching mechanism $\Sigma$, migration mechanism $\Psi$ (and initial value $z$ that of $Z$), which is non-explosive  because $\Sigma$ is (sub)critical \cite[Corollary~2.1]{vidmar2021continuousstate}. The proof of Theorem~\ref{thm:time-change} is completed by pathwise arguments to go back from $Y$ to $Z$.
\end{proof}
 When $-\Psi$ is the Laplace exponent of a subordinator, the CBM process $Y$ of Theorem~\ref{thm:time-change} is a CBI process with immigration mechanism $-\Psi$, stopped at its first hitting time of $0$. \\\\
\noindent \textbf{Proof of Proposition~\ref{propositionstatespace}}:\label{proof-2.2} Recall   that by definition $z^*=(\limsup_{u\to\infty} \frac{-\Psi(u)}{\Sigma(u)})\vee 0$.  Notice first that $z^*>0$ if and only if we are in the subordinator case with $\mu:=\underset{u\rightarrow \infty}{\lim}\frac{-\Psi(u)}{u}>0$ and $\Sigma$ is of finite variation type, i.e. $D:=\underset{u\rightarrow \infty}{\lim} \frac{\Sigma(u)}{u}<\infty$ (the two limits exist a priori in $[0,\infty)$ and $(0,\infty]$, respectively), in which case, $z^*=\limsup_{u\to\infty} \frac{-\Psi(u)}{\Sigma(u)}=\lim_{u\to \infty} \frac{-\Psi(u)}{\Sigma(u)}=\frac{\mu}{D}$. Assume now $z^*>0$. According to \cite[Proposition~5]{Duhalde}, when $Y_0=z>0$, then a.s. $Y_t\geq e^{-Dt}z +z^*(1-e^{-Dt})$, which is $>z^*$ for all $t\in [0,\infty)$ or is $\geq z$ and tending to $z^*$ as $t\to\infty$ according as to whether $z>z^*$ or $z\in (0,z^*]$; here $Y$ is as in Theorem~\ref{thm:time-change}.  Hence, by the time-change representation of the CBC process $Z$, one also has a.s. $Z_t>z^*$ for all $t\in [0,\infty)$  as soon as its starting value $z$ lies in $(z^*,\infty)$ and similarly one has a.s. $Z_t\geq z$ for all $t\in [0,\infty)$ with $\liminf_{t\to\infty}Z_t\geq z^*$ in case $z\in (0,z^*]$. \qed

\section{Attraction to the boundaries: proof of Theorem \ref{attractiveboundaries}}\label{sectionproofattraction} 

The proof of Theorem~\ref{attractiveboundaries} is based on the time-change representation of CBCs via CBMs. Let then, for the purposes of this section, $Y=(Y_u)_{u\in [0,\zeta)}$ be the CBM of Theorem~\ref{thm:time-change}.
We state several lemmas, the combined conclusion of which will be Theorem~\ref{attractiveboundaries}. Recall from the beginning  of Subsection~\ref{subsection:classification} that $x_0$ is a fixed strictly positive real number and the notation \eqref{page:S_V}-\eqref{eq:SZ-def}.

\begin{lemma}\label{lemmaexitprobai} Let  $z^*<\mathsf{a}<z<\infty$. If $S_V(0,x_0]<\infty$ then  $\mathbb{P}_z(\zeta^-_\mathsf{a}<\zeta_\infty)=\frac{S_Z(z)}{S_Z(\mathsf{a})}\in (0,1)$.
If  $S_V(0,x_0]=\infty$ then $\mathbb{P}_z(\zeta^-_\mathsf{a}<\zeta_\infty)=1$.
\end{lemma}
\begin{proof}
Set $\sigma_\mathsf{a}:=\inf\{t\in [0,\zeta): Y_t\leq \mathsf{a}\}$. By \cite[Theorem~3.1]{vidmar2021continuousstate} for the non-subordinator case, respectively by \cite[Theorem~1]{Duhalde} for the subordinator case, we have for any $\theta\in (0,\infty)$,
\begin{equation}\label{LTsigmaa}\mathbb{E}_z[e^{-\theta \sigma_\mathsf{a}}]=\frac{\Phi_{\theta}(z)}{\Phi_{\theta}(\mathsf{a})},
\end{equation}
where \[\Phi_{\theta}(z):=\int_{0}^{\infty}\frac{\ddr x}{\Sigma(x)}e^{-zx-\int_{x_0}^{x}\frac{\Psi(u)-\theta}{\Sigma(u)}\ddr u}=\int_0^{\infty}S_V(\ddr x)e^{-zx+\int_{x_0}^{x}\frac{\theta}{\Sigma(u)}\ddr u}.\]
Therefore, by the time-change representation:
\begin{equation}\label{exitbya} \mathbb{P}_z(\zeta^-_\mathsf{a}<\zeta_\infty)=\mathbb{P}_z(\sigma_\mathsf{a}<\infty)=\underset{\theta \rightarrow 0+}{\lim}\frac{\Phi_{\theta}(z)}{\Phi_{\theta}(\mathsf{a})}=\underset{\theta \rightarrow 0+}{\lim}\frac{\int_{0}^{\infty}\frac{\ddr x}{\Sigma(x)} e^{-zx-\int_{x_0}^{x}\frac{\Psi(u)-\theta}{\Sigma(u)}\ddr u}}{\int_{0}^{\infty}\frac{\ddr x}{\Sigma(x)} e^{-\mathsf{a}x-\int_{x_0}^{x}\frac{\Psi(u)-\theta}{\Sigma(u)}\ddr u}}.
\end{equation}
Assume first $S_V(0,x_0]=\int_{0}^{x_0}\frac{\ddr x}{\Sigma(x)}e^{-\int_{x_0}^{x}\frac{\Psi(u)}{\Sigma(u)}\ddr u}<\infty$. Then, since $z>\mathsf{a}>z^*$,
\[\int_0^{\infty}e^{-zx} S_V(\ddr x)<\int_0^{\infty}e^{-\mathsf{a}x} S_V(\ddr x)<\infty.\]
Moreover, splitting the integrals in \eqref{exitbya} in two pieces, according to the domains $(0,x_0]$ and $(x_0,\infty)$, and applying monotone convergence  on $(0,x_0]$ and dominated convergence on $(x_0,\infty)$, we get the convergence as $\theta$ goes to $0$ of the right-hand side of \eqref{exitbya} and obtain
\begin{equation}\label{convh_theta}
\mathbb{P}_z(\zeta^-_\mathsf{a}<\zeta_\infty)=\frac{\int_{0}^{\infty}\frac{\ddr x}{\Sigma(x)} e^{-zx-\int_{x_0}^{x}\frac{\Psi(u)}{\Sigma(u)}\ddr u}}{\int_{0}^{\infty}\frac{\ddr x}{\Sigma(x)} e^{-\mathsf{a}x-\int_{x_0}^{x}\frac{\Psi(u)}{\Sigma(u)}\ddr u}}=\frac{S_Z(z)}{S_Z(\mathsf{a})}\in (0,1).
\end{equation}

\noindent Assume now $S_V(0,x_0]=\infty$. Then we see from \eqref{exitbya} that 
\begin{align*}
\mathbb{P}_z(\zeta^-_\mathsf{a}<\zeta_\infty)&\geq \underset{\theta \rightarrow 0+}{\lim} \frac{\int_0^{x_0}e^{-zx}\frac{1}{\Sigma(x)}e^{-\int_{x_0}^{x}\frac{\Psi(u)-\theta}{\Sigma(u)}\ddr u}\ddr x}{\int_0^{x_0}e^{-\mathsf{a}x}\frac{1}{\Sigma(x)}e^{-\int_{x_0}^{x}\frac{\Psi(u)-\theta}{\Sigma(u)}\ddr u}\ddr x+\int_{x_0}^{\infty} e^{-\mathsf{a}x}\frac{1}{\Sigma(x)}e^{-\int_{x_0}^{x}\frac{\Psi(u)-\theta}{\Sigma(u)}\ddr u}\ddr x}\\
&\geq \underset{\theta \rightarrow 0+}{\lim} \frac{e^{-zx_0}\int_0^{x_0}\frac{1}{\Sigma(x)}e^{-\int_{x_0}^{x}\frac{\Psi(u)-\theta}{\Sigma(u)}\ddr u}\ddr x}{\int_0^{x_0}\frac{1}{\Sigma(x)}e^{-\int_{x_0}^{x}\frac{\Psi(u)-\theta}{\Sigma(u)}\ddr u}\ddr x+\int_{x_0}^{\infty} e^{-\mathsf{a}x}\frac{1}{\Sigma(x)}e^{-\int_{x_0}^{x}\frac{\Psi(u)-\theta}{\Sigma(u)}\ddr u}\ddr x}\\
&=\underset{\theta \rightarrow 0+}{\lim}\frac{e^{-zx_0}}{1+\int_{x_0}^{\infty} e^{-\mathsf{a}x}\frac{1}{\Sigma(x)}e^{-\int_{x_0}^{x}\frac{\Psi(u)-\theta}{\Sigma(u)}\ddr u}\ddr x/\int_0^{x_0}\frac{1}{\Sigma(x)}e^{-\int_{x_0}^{x}\frac{\Psi(u)-\theta}{\Sigma(u)}\ddr u}\ddr x}.
\end{align*}
By monotone convergence $$\int_0^{x_0} \frac{1}{\Sigma(x)}e^{-\int_{x_0}^{x}\frac{\Psi(u)-\theta}{\Sigma(u)}\ddr u}\ddr x \underset{\theta \rightarrow 0+}{\longrightarrow} S_V(0,x_0]=\infty,$$ while by dominated convergence
$$\int_{x_0}^{\infty} e^{-\mathsf{a}x}\frac{1}{\Sigma(x)}e^{-\int_{x_0}^{x}\frac{\Psi(u)-\theta}{\Sigma(u)}\ddr u}\ddr x \underset{\theta \rightarrow 0+}{\longrightarrow}\int_{x_0}^{\infty} e^{-\mathsf{a}x}\frac{1}{\Sigma(x)}e^{-\int_{x_0}^{x}\frac{\Psi(u)}{\Sigma(u)}\ddr u}\ddr x <\infty;$$
and we conclude that 
$\mathbb{P}_z(\zeta^-_\mathsf{a}<\zeta_\infty)\geq e^{-zx_0}$. Since $x_0$ can be chosen arbitrarily small, we get finally $\mathbb{P}_z(\zeta^-_\mathsf{a}<\zeta_\infty)=1$.
\end{proof}

\begin{lemma}\label{lemmacvii} Let $z\in (z^*,\infty)$. If $\Psi\ne 0$, then the following equivalences hold:
$Z_t\underset{t\rightarrow \infty}{\longrightarrow} 0$ with positive $\PP_z$-probability (respectively $\PP_z$-almost surely) if and only if $S_V(x_0,\infty)<\infty$ (respectively $S_V(x_0,\infty)<\infty$ and $S_V(0,x_0]=\infty$); when $S_V(0,\infty)<\infty$, then, moreover, $\mathbb{P}_z(Z_t\underset{t\rightarrow \infty}{\longrightarrow} 0)=\frac{S_Z(z)}{S_Z(0)}\in (0,1)$, where $S_Z$ is as in \eqref{eq:SZ-def}.  If $\Psi= 0$, then $\PP_z$-almost surely $Z_t\underset{t\rightarrow \infty}{\longrightarrow} 0$.
\end{lemma}
\begin{proof} 
We continue to use $\sigma_0$ to denote the first hitting time of $0$ by  the CBM $Y$ of Theorem \ref{thm:time-change}. 

If $\Psi= 0$, then $Y$ is a CB process with branching mechanism $\Sigma$. 
Since $\Sigma$ is (sub)critical, one then has $Y_t\underset{t\rightarrow \infty}{\longrightarrow} 0$ a.s.-$\PP_z$, always, see e.g. \cite[Theorem 12.7]{Kyprianoubook}, therefore $Z_t\underset{t\rightarrow \infty}{\longrightarrow} 0$ a.s.-$\PP_z$. 

Assume now $\Psi\ne 0$. 

CBMs and CBIs (that are not CBs) cannot converge towards $0$ without hitting it (said another way, they do not extinguish), i.e. we have that $\{Y_t\underset{t\rightarrow \infty}{\longrightarrow} 0\}=\{\sigma_0<\infty\}$ a.s.. For CBMs (that are not CBIs) this is noted in \cite[Corollary~3.1]{vidmar2021continuousstate}, for CBIs (that are not CBs), it follows at once from \cite[Eq.~(18)]{Duhalde} which states that such a process has infinite superior limit. 
One has therefore the following almost sure equality of events: $$\{Z_t\underset{t\rightarrow \infty}{\longrightarrow} 0\}=\{\sigma_0<\infty\}.$$ 

By letting $\mathsf{a}$ go to $0$ in \eqref{LTsigmaa} when $z^*=0$, trivially by Proposition~\ref{propositionstatespace} for the case when $z^*>0$, we see that 
\begin{equation}\label{eq:hitting-zero}
\mathbb{E}_z[e^{-\theta \sigma_0}]=\frac{\Phi_{\theta}(z)}{\Phi_{\theta}(0)},
\end{equation}
and $\Phi_{\theta}(0)=\int_{0}^{\infty}\frac{\ddr x}{\Sigma(x)}e^{-\int_{x_0}^{x}\frac{\Psi(u)-\theta}{\Sigma(u)}\ddr u}<\infty$ if and only if $S_V(x_0,\infty)=\int_{x_0}^{\infty}\frac{\ddr x}{\Sigma(x)}e^{-\int_{x_0}^{x}\frac{\Psi(u)}{\Sigma(u)}\ddr u}<\infty$ (note that $S_V(x_0,\infty)<\infty$ entails $\int_{x_0}^{\infty}\frac{\ddr u}{\Sigma(u)}<\infty$ which in turn ensures that $\Phi_\theta(0)<\infty$ when  $S_V(x_0,\infty)<\infty$).  Thus $\sigma_0<\infty$ with positive $\PP_z$-probability if and only if $S_V(x_0,\infty)<\infty$.  
%
This establishes the equivalence for convergence towards $0$ with positive probability. For the almost sure convergence,  we may and do assume $S_V(x_0,\infty)<\infty$. Then the same reasoning as in the proof of Lemma \ref{lemmaexitprobai} yields
\[\mathbb{P}_z(\sigma_0<\infty)\geq \underset{\theta \rightarrow 0+}{\lim}\frac{e^{-x_0z}}{1+\int_{x_0}^{\infty} \frac{1}{\Sigma(x)}e^{-\int_{x_0}^{x}\frac{\Psi(u)-\theta}{\Sigma(u)}\ddr u}\ddr x/\int_0^{x_0}\frac{1}{\Sigma(x)}e^{-\int_{x_0}^{x}\frac{\Psi(u)-\theta}{\Sigma(u)}\ddr u}\ddr x}.\]
If $S_V(0,x_0]=\infty$ the denominator above converges to $1$ as $\theta \to 0+$ and we have $\mathbb{P}_z(\sigma_0<\infty)\geq e^{-x_0z}$; since $x_0$ can be chosen arbitrarily small we get $\mathbb{P}_z(\sigma_0<\infty)=1$. Conversely, if $S_V(0,x_0]<\infty$, i.e. (together with  $S_V(x_0,\infty)<\infty$) $S_V(0,\infty)<\infty$, then by dominated convergence in \eqref{eq:hitting-zero} we obtain $\mathbb{P}_z(Z_t\underset{t\rightarrow \infty}{\longrightarrow} 0)=\frac{S_Z(z)}{S_Z(0)}$.
\end{proof}
\begin{lemma}\label{lemmacvi} 
Let $z\in (z^*,\infty)$. If $S_V(0,x_0]<\infty$, then $$\PP_z(\lim_{t\to\infty}Z_t=\infty)+\PP_z(\lim_{t\to\infty}Z_t=0)=1.$$ Furthermore,  $Z_t\underset{t\rightarrow \infty}{\longrightarrow} \infty \text{ with positive $\PP_z$-probability}$ (respectively, $\PP_z$-almost surely) if and only if $S_V(0,x_0]<\infty$ (respectively, $S_V(0,x_0]<\infty$ and $S_V(x_0,\infty)=\infty$).
\end{lemma}
\begin{proof}
In the non-subordinator case the first claim is immediate by time-change and the fact that for the CBM $Y$, $\lim_{s\to\infty}Y_s=\infty$ a.s. on $\{\sigma_0=\infty\}$ \cite[Corollary~3.1]{vidmar2021continuousstate}. In the subordinator case the condition  $S_V(0,x_0]<\infty$ ensures that the underlying unstopped [not stopped on hitting $0$] CBI process is almost surely transient (i.e. tends to $\infty$) \cite[Theorem~3(a)]{Duhalde} and again the first claim follows by time-change.

As for the second statement, we treat the non-respective case first. When the process escapes to $\infty$ with positive $\PP_z$-probability, the $\PP_z$-probability of staying above level $\mathsf{a}$ is positive for some $\mathsf{a}\in (z^*,\infty)$ and $S_V(0,x_0]$ has to be finite, since otherwise by Lemma~\ref{lemmaexitprobai} the latter probability would be zero. Conversely, if $S_V(0,x_0]<\infty$, then  $\PP_z(Z_t\underset{t\rightarrow \infty}{\longrightarrow} \infty)>0$ follows by combining the first statement with Lemma~\ref{lemmacvii} (and the fact that  $S_V(0,x_0]=\infty$ when $\Psi=0$). 

The respective case of the second statement follows easily from the non-respective one, the first statement and from Lemma~\ref{lemmacvii}.
\end{proof}

\noindent \textbf{Proof of Theorem \ref{attractiveboundaries}}: Statements \eqref{attractiveboundaries:i} and \eqref{attractiveboundaries:ii}  follow from Lemma~\ref{lemmaexitprobai}. Parts \eqref{attractiveboundaries:iii}-\eqref{attractiveboundaries:v} follow from Lemmas~\ref{lemmacvii}-\ref{lemmacvi}. \qed

\section{Study of non-explosion,  first passage times \& extinction}\label{sectionstudyfirstpassagetime}
\subsection{A sufficient condition for non-explosion: proof of Proposition \ref{propositionsuffcond}}
We know already that if $S_V(0,x_0]=\infty$ then $\infty$ is not attracting for $Z$ and therefore $Z$ does not explode, see Remark~\ref{H(0)finite}. In particular, non-explosion is guaranteed if $\Psi$ is (sub)critical, since then indeed $S_V(0,x_0]=\infty$. We finish the proof by establishing through a series of lemmas that when $\Psi'(0+)\in [-\infty,0)$ (supercritical $\Psi$) and $\int_{0+}\frac{\ddr x}{-\Psi(x)}=\infty$, then the \CBC cannot explode. 

The first lemma provides an increasing invariant function for supercritical CBs. We state it separately as it can be of independent interest for other generalisations of CBs. Call  $\mathscr{L}^{\mathrm{b}}$ the generator of the CB$(\Psi)$, viz. for  $f\in C^2_b([0,\infty))$, $$\mathscr{L}^{\mathrm{b}}f(z):=z\mathrm{L}^{\Psi}f(z),\quad z\in [0,\infty).$$ Assume $\Psi'(0+)\in [-\infty,0)$ and put $$\rho:=\sup\{x\in (0,\infty): \Psi(x)<0\}\in (0,\infty].$$ Pick $x_0\in (0,\rho)$ and let $\theta\in (0,\infty)$. Set \begin{equation}\label{eigenfunctionCB} \bar{f}_\theta(z):=
\int_0^{\rho}(1-e^{-xz})\frac{\theta}{-\Psi(x)}e^{\int_x^{x_0}\frac{\theta}{-\Psi(u)}\ddr u}\ddr x
\in [0,\infty],\quad z\in [0,\infty).
\end{equation}

\begin{lemma}[Increasing eigenfunction of CB$(\Psi)$] \label{lemma:increaisng-for-cb}
Assume $\theta\in (0,-\Psi'(0+))$. Then 
\begin{equation}\label{eigenfunctionCB-bis}
\bar{f}_\theta(z)=z\int_0^\rho e^{-zx}e^{\int_x^{x_0}\frac{\theta}{-\Psi(u)}\ddr u}\ddr x<\infty,\quad z\in [ 0,\infty);
\end{equation}
 furthermore,
\begin{enumerate}[(i)]
\item\label{claim:1} $\bar{f}_\theta$ is an increasing solution to $\mathscr{L}^{\mathrm{b}}\bar{f}_\theta=\theta \bar{f}_\theta$ and
\item\label{claim:2}  $\bar{f}_\theta$ is bounded if and only if $\int_{0+}\frac{\ddr u}{-\Psi(u)}<\infty$.
\end{enumerate}
\end{lemma}
\begin{proof}
First we check that $\bar{f}_\theta(z)<\infty$ for all $z\in [ 0,\infty)$. Recall  that $1-e^{-zx}\leq (zx)\land 1$ for $x\in [0,\infty)$ and that $\frac{-\Psi(x)}{x}\underset{x\rightarrow 0}{\longrightarrow} -\Psi'(0+)\in (0,\infty]$. Let $\mathsf{c}\in (\theta,-\Psi(0+))$; there exists then a choice for $x_0$ (which we may vary to our convenience, changing $\bar{f}_\theta$ only by a multiplicative constant) close enough to $0$ such that for all $u\in (0, x_0]$, $\frac{-\Psi(u)}{u}\geq \mathsf{c}$, thus $\frac{u}{-\Psi(u)}\leq \frac{1}{\mathsf{c}}$ and 
\begin{align*}
\bar{f}_\theta(z)&\leq \theta\int_{0}^{x_0}\frac{xz}{-\Psi(x)}e^{\int_x^{x_0}\frac{\theta}{\mathsf{c}u}\ddr u}\dd x+\int_{x_0}^\rho \frac{\theta}{-\Psi(x)}e^{\int_x^{x_0}\frac{\theta}{-\Psi(u)}\ddr u}\ddr x\\
&\leq \theta \frac{z}{\mathsf{c}}\int_0^{x_0} \left(\frac{x_0}{x}\right)^{\theta/\mathsf{c}}\ddr x-e^{-\int_{x_0}^x \frac{\theta}{-\Psi(u)}\ddr u}\vert_{x=x_0}^{x=\rho}<\infty, \text{ since } \theta/\mathsf{c}<1. 
\end{align*}
\eqref{claim:1}. 
It is plain that $\bar{f}_\theta$ is increasing. Notice that  $\mathscr{L}^{\mathrm{b}}_z(1-e^{-zx})=-z\Psi(x)e^{-zx}$ for $x\in [0,\infty)$, $z\in [0,\infty)$. Differentiation under the integral sign and Tonelli's theorem, then integration by parts, yield 
\begin{align*}
\mathscr{L}^{\mathrm{b}}\bar{f}_\theta(z)&=\int_0^\rho z(-\Psi(x))e^{-zx}\frac{\theta}{-\Psi(x)}e^{\int_x^{x_0}\frac{\theta}{-\Psi(u)}\ddr u}\ddr x=\theta\int_0^\rho z e^{-zx}e^{\int_x^{x_0}\frac{\theta}{-\Psi(u)}\ddr u}\ddr x\\
&=\left(\theta (1-e^{-zx})e^{\int_x^{x_0}\frac{\theta}{-\Psi(u)}\ddr u}\right)\vert_{x=0}^\rho+\theta \bar{f}_\theta(z)=\theta \bar{f}_\theta(z),
\end{align*}
where the last equality uses again the estimates $1-e^{-zx}\leq (zx)\land 1$ and $e^{\int_x^{x_0}\frac{\theta}{-\Psi(u)}\ddr u}\leq \left(\frac{x_0}{x}\right)^{\theta/\mathsf{c}}$, so that $\lim_{x\downarrow 0}(1-e^{-zx})e^{\int_x^{x_0}\frac{\theta}{-\Psi(u)}\ddr u}=0$, but also the fact that $\Psi(x)$ behaves like a linear function vanishing at $\rho$ and with strictly positive slope around $\rho$ when $\rho <\infty$, respectively that  $-\Psi$ is bounded in linear growth when $\rho=\infty$, which renders $\lim_{x\uparrow \rho}(1-e^{-zx})e^{\int_x^{x_0}\frac{\theta}{-\Psi(u)}\ddr u}=0$. En route we have checked the equality in \eqref{eigenfunctionCB-bis}.

\eqref{claim:2}. Note that by definition, as $z$ goes to $\infty$, $\bar{f}_\theta(z)$ tends by monotone convergence to \[\int_0^\rho\frac{\theta}{-\Psi(x)}e^{\int_x^{x_0}\frac{\theta}{-\Psi(u)}\ddr u}\ddr x=-e^{-\int_{x_0}^x \frac{\theta}{-\Psi(u)}\ddr u}\vert_{x=0}^{x=\rho}=e^{\int_0^{x_0}\frac{\theta}{-\Psi(u)}\ddr u} \in (0,\infty].\] 
So $\bar{f}_\theta$ is bounded or unbounded according as to whether $\int_{0+}\frac{\ddr u}{-\Psi(u)}<\infty$ or not.
\end{proof}
We now return to CBCs. Recall $\mathscr{L}$ of \eqref{genZ}.
\begin{lemma}\label{LyapounovCBC} Assume $\theta\in (0,-\Psi'(0+))$. Then $\mathscr{L}\bar{f}_\theta\leq \theta \bar{f}_\theta$.
\end{lemma}
\begin{proof} Set $\mathscr{L}^{\mathrm{c}}f(z):=z^2\mathrm{L}^{\Sigma}f(z)$ for  $f\in C^2_b([0,\infty)$, $z\in [0,\infty)$, so that $\mathscr{L}=\mathscr{L}^{\mathrm{c}}+\mathscr{L}^{\mathrm{b}}$. For $z\in [0,\infty)$  we estimate
\begin{align*}
\mathscr{L}^\mathrm{c}\bar{f}_\theta(z)&=\int_0^{\rho}\mathscr{L}_z^c(1-e^{-zx})\frac{\theta}{-\Psi(x)}e^{\int_x^{x_0}\frac{\theta}{-\Psi(u)}\ddr u}\ddr x\\
&=\int_0^{\rho}(-z^2\Sigma(x)e^{-zx})\frac{\theta}{-\Psi(x)}e^{\int_x^{x_0}\frac{\theta}{-\Psi(u)}\ddr u}\ddr x\leq 0.
\end{align*}
Thus $\mathscr{L}\bar{f}_\theta=\mathscr{L}^\mathrm{c}\bar{f}_\theta+\mathscr{L}^\mathrm{b}\bar{f}_\theta\leq \mathscr{L}^\mathrm{b}\bar{f}_\theta=\theta\bar{f}_\theta$ on using Lemma~\ref{lemma:increaisng-for-cb}.
\end{proof}
The next lemma concludes the argument for Proposition \ref{propositionsuffcond}.
\begin{lemma}\label{nonexplosionCB:CBC} Assume $\Psi'(0+)\in [-\infty,0)$. If $\int_{0+}\frac{\ddr u}{-\Psi(u)}=\infty$  then the CBC$(\Sigma,\Psi)$ process $Z$ does not explode.
\end{lemma}
\begin{proof}
Pick a  $\theta\in (0,-\Psi'(0+))$. Since $\int_{0+}\frac{\ddr x}{-\Psi(x)}=\infty$, we have  $\text{$\uparrow$-$\lim$}_{z\to\infty}\bar f_\theta(z)=\infty$. Fix also an $r\in (0,\infty)$. For $\mathsf{c}\in [r,\infty)$ let $\bar{f}^\mathsf{c}_\theta$ be any  nonnegative $C^2_b([0,\infty))$ function which agrees with $\bar f_\theta$ on $[0,\mathsf{c})$ and minorizes  $\bar f_\theta$  everywhere,  e.g. one such function is obtained by taking any nonnegative $C^1([0,\infty))$ map $h$ that agrees with ${\bar f_\theta}'$ on $[0,\mathsf{c})$, minorizes ${\bar f_\theta}'$ everywhere and which vanishes on a neighborhood of infinity (clearly it exists), and then putting $\bar{f}^\mathsf{c}_\theta(x):=\bar{f}_\theta(\mathsf{c})+\int_{\mathsf{c}}^x h(y)\dd y$, $x\in [\mathsf{c},\infty)$ (note that $\bar f_\theta$ itself is nonnegative  $C^2([0,\infty))$, which follows by differentiation under the integral sign in \eqref{eigenfunctionCB}). 
Since $\bar f_\theta\geq \bar f^\mathsf{c}_\theta$ for the first inequality, and by Lemma \ref{LyapounovCBC} for the second,
$$
\mathscr{L} \bar{f}^\mathsf{c}_\theta(z)\leq\mathscr{L}\bar{f}_\theta(z)\leq \theta \bar{f}_\theta(z)=\theta \bar{f}^\mathsf{c}_\theta(z),\quad z\in [0,r).
$$
 Taking into account that $\bar f_\theta^{\mathsf{c}}$ is bounded, it follows from the statement surrounding \eqref{cbc-construction:i:a} that the process $(e^{-\theta (t\wedge \zeta_r^+)}\bar f_\theta^{\mathsf{c}}(Z_{t\wedge \zeta_r^+}),t\geq 0)$ is a supermartingale; hence, for all $t\in [0,\infty)$,
\[\mathbb{E}_z[e^{-\theta (t\wedge \zeta_r^+)}\bar{f}_\theta^\mathsf{c}(Z_{t\wedge \zeta_r^+})]\leq \bar{f}_\theta^\mathsf{c}(z).\]
Let now $\mathsf{c}\to \infty$, consider the event $\{\zeta_r^+\leq t\}$ and recall that $\bar{f}_\theta$ is nondecreasing; one gets
\[\bar{f}_\theta(z)\geq \mathbb{E}_z[e^{-\theta (t\wedge \zeta_r^+)}\bar{f}_\theta(Z_{t\wedge \zeta_r^+})]\geq \mathbb{E}_z[e^{-\theta \zeta_r^+}\bar{f}_\theta(Z_{\zeta_r^+})\mathbbm{1}_{\{\zeta_r^+\leq t\}}]\geq \bar{f}_\theta(r)\mathbb{E}_z[e^{-\theta \zeta_r^+}\mathbbm{1}_{\{\zeta_r^+\leq t\}}].\]
Letting next $t\to \infty$ we get the bound
$\mathbb{E}_z[e^{-\theta \zeta_r^+}]\leq \frac{\bar{f}_\theta(z)}{\bar{f}_\theta(r)}$. Effecting finally the limit $r\to\infty$  yields
\[\mathbb{E}_z[e^{-\theta \zeta_\infty}]\leq \underset{r\rightarrow \infty}{\lim} \frac{\bar{f}_\theta(z)}{\bar{f}_\theta(r)}=0,\] 
which means that $\zeta_\infty=\infty$ a.s., as required.
\end{proof}
\subsection{A decreasing eigenfunction of $Z$: proof of Theorem \ref{firstpassagetimestheorem}}
The proof will again proceed in several steps. 
We start by linking  nondecreasing eigenfunctions of $\mathscr{G}$ to  decreasing ones for $\mathscr{L}$. Recall the form \eqref{genZ} of $\mathscr{L}$, the Laplace duality $\mathscr{L}_ze^{-vz}=z^{2}\Sigma(v)e^{-zv}+z\Psi(v)e^{-zv}=\mathscr{A}_ve^{-zv}$, and the action $\mathscr{G}h=\Sigma h''+(\Sigma'+\Psi)h'=(\Sigma h')'+\Psi h'$. Observe also that the equation $\mathscr{G}h=\theta h$ admits at least one strictly increasing solution $h:(0,\infty)\to (0,\infty)$, see e.g. Mandl \cite[\#3, Chapter II, page 28]{zbMATH03287297}.
\begin{lemma}[A decreasing eigenfunction]\label{proposition:h-to-f}
Let $\theta\in (0,\infty)$, and suppose $h_\theta\in C^2((0,\infty))$ is nonnegative, not zero, nondecreasing and satisfies $\mathscr{G} h_\theta=\theta h_\theta$ on $(0,\infty)$. Put 
\begin{equation}\label{eq:from h-to-f}
f_\theta(z):=z\int_0^\infty e^{-zv}h_\theta(v)\dd v=h_\theta(0+)+\int_0^\infty e^{-zv}h'_\theta(v)\dd v,\quad z\in (0,\infty).
\end{equation}
 Then $\mathscr{L} f_\theta=\theta f_\theta$ on the interior of $\{f_\theta<\infty\}$.
 \end{lemma}
\begin{proof}
The equality  in \eqref{eq:from h-to-f} follows by Tonelli. For $z$ from the interior of $\{f_\theta<\infty\}$,  we compute by differentiating under the integral sign and using Tonelli, then via per partes:
\label{page:bigdisplay}
 \begin{align*}
\mathscr{L} f_\theta(z)&=\mathscr{L}_z\int_0^\infty e^{-zv}h_\theta'(v)\dd v=\int_0^\infty \mathscr{L}_z e^{-zv}h_\theta'(v)\dd v=\int_0^\infty (z^2\Sigma(v) e^{-zv}+z\Psi(v)e^{-zv})h_\theta'(v)\dd v\\
&=\lim_{\epsilon\downarrow 0}\Sigma(\epsilon)h_\theta'(\epsilon)ze^{-\epsilon z}+ \lim_{n\to\infty} -\Sigma(n)h_\theta'(n)ze^{-zn}\\
&\quad +\int_{\epsilon}^n\left (\frac{\dd}{\dd v}\left(\Sigma(v)h_\theta'(v)\right)+\Psi(v)h_\theta'(v)\right)ze^{-zv}\dd v\\
&=\lim_{\epsilon\downarrow 0}\Sigma(\epsilon)h_\theta'(\epsilon)z e^{-\epsilon z}+ \lim_{n\to\infty} -\Sigma(n)h_\theta'(n)ze^{-zn}+\int_{\epsilon}^n \mathscr{G} h_\theta(v)ze^{-zv}\dd v\\
&=\lim_{\epsilon\downarrow 0}\Sigma(\epsilon)h_\theta'(\epsilon)ze^{-\epsilon z}+ \lim_{n\to\infty} -\Sigma(n)h_\theta'(n)ze^{-zn}+\theta \int_{\epsilon}^n  h_\theta(z)ze^{-zv}\dd v\\
&=\theta z \int_0^\infty  h_\theta(z)e^{-zv}\dd v+\lim_{\epsilon\downarrow 0}\Sigma(\epsilon)h_\theta'(\epsilon)ze^{-\epsilon z}+ \lim_{n\to\infty} -\Sigma(n)h_\theta'(n)ze^{-zn}\\
&=\theta f_\theta(z)+\lim_{\epsilon\downarrow 0}\Sigma(\epsilon)h_\theta'(\epsilon)ze^{-\epsilon z}+ \lim_{n\to\infty} -\Sigma(n)h_\theta'(n)ze^{-zn}.
\end{align*}
 In particular the limits $\lim_{n\to\infty} -\Sigma(n)h'_\theta(n)ze^{-zn}$ and $\lim_{\epsilon\downarrow 0}\Sigma(\epsilon)h'_\theta(\epsilon)ze^{-\epsilon z}$  both exist in $\mathbb{R}$ for all $z$ from the interior of $\{f_\theta<\infty\}$. The first limit must in fact be zero, since such $z$  can always be made a little smaller. As for the second limit, it is (modulo $z$) $\lim_{\epsilon\downarrow 0}\Sigma(\epsilon) h'_\theta(\epsilon)$.  Suppose per absurdum that this limit is not zero, hence, from $(0,\infty)$. Since $\int_{0+}\frac{1}{\Sigma(x)}\ddr x=\infty$, we see that $\int_{0+}h_\theta’(x)\ddr x$ diverges, thus $h_\theta(x)$ is infinite for all $x>0$,  which is a contradiction. \end{proof}
 
We now check that the function defined in \eqref{eq:from h-to-f} is finite on $(z^*,\infty)$, where we may recall from \eqref{eq:z-star} that $z^*=\left(\limsup_{u\to\infty} \frac{-\Psi(u)}{\Sigma(u)}\right)\vee 0<\infty$  and that by Proposition~\ref{propositionstatespace} actually $z^*=0$ except possibly in the subordinator case. 

\begin{lemma}\label{lemma:f-theta-finite}
For all  $\theta\in (0,\infty)$ the function $f_\theta$ in  \eqref{eq:from h-to-f} is finite on  $ (z^*,\infty)$.
\end{lemma}
\begin{proof}
Write $h:=h_\theta$ for short and consider $g:=\Sigma h'$. Let $z\in (z^*,\infty)$. Pick  a $\mathsf{c}\in(0,\infty)$ such that $A_\mathsf{c}+\sqrt{\frac{\theta}{\Sigma(\mathsf{c})}}<z$, where $A_\mathsf{c}:=(\sup_{u\in [c,\infty)}\frac{-\Psi(u)}{\Sigma(u)})\lor 0$. We have $g'=-\frac{\Psi}{\Sigma}g+\theta \int_\mathsf{c}^\cdot  \frac{g}{\Sigma}+\theta h(\mathsf{c})\leq A_\mathsf{c}g+\frac{\theta}{\Sigma(\mathsf{c})}\int_\mathsf{c}^\cdot g+\theta h(\mathsf{c})$ on $[\mathsf{c},\infty)$. Then let $\tilde g$ be the $C^2([\mathsf{c},\infty))$ solution to $\tilde g'= A_\mathsf{c}\tilde g+\frac{\theta}{\Sigma(\mathsf{c})}\int_\mathsf{c}^\cdot\tilde g+\theta h(\mathsf{c})$ with initial condition $\tilde g(\mathsf{c})=g(\mathsf{c})+1$; in other words, solution of the  second order o.d.e. with constant coefficients $\tilde g''=A_\mathsf{c}\tilde g'+\frac{\theta}{\Sigma(\mathsf{c})}\tilde g$, $\tilde g(\mathsf{c})=g(\mathsf{c})+1$, $\tilde g'(\mathsf{c})=A_\mathsf{c}\tilde g(\mathsf{c})+\theta h(\mathsf{c})$. The function $\tilde g$ is a linear combination of (at most) two exponentials with absolute rate $\leq A_\mathsf{c}+\sqrt{\frac{\theta}{\Sigma(\mathsf{c})}}<z$ (using the elementary estimate $\sqrt{\mathsf{a}^2+\mathsf{b}^2}\leq \mathsf{a}+\mathsf{b}$, $\{\mathsf{a},\mathsf{b}\}\subset [0,\infty)$, to get a bound on the roots of the characteristic polynomial). Furthermore $\zeta:=\tilde g-g$ satisfies $\zeta(\mathsf{c})=1$ and $\zeta'\geq A_\mathsf{c}\zeta+\frac{\theta}{\Sigma(\mathsf{c})}\int_\mathsf{c}^\cdot\zeta$; therefore $\zeta\geq 0$ (even $\geq 1$), i.e. $g\leq \tilde g$ throughout $[\mathsf{c},\infty)$. Consequently $h'=\frac{g}{\Sigma}\leq \Sigma(\mathsf{c})^{-1}g\leq \Sigma(\mathsf{c})^{-1}\tilde g$ on $[\mathsf{c},\infty)$. The derivative of $h'$ being bounded (up to a multiplicative constant) on a neighborhood of infinity with an  exponential of rate $<z$, the same is true of $h$ itself. The claim follows.
\end{proof}
Under the assumption of non-explosion the next lemma characterizes the Laplace transforms of the first-passage times via the maps $f_\theta$, $\theta\in (0,\infty)$.
\begin{lemma}\label{lemmaLT} Assume that the process $Z$ does not explode. Let $\theta\in (0,\infty)$ and let $f_\theta$ be defined as in \eqref{eq:from h-to-f}.  Then for $\mathsf{a}\leq z$ from $(z^*,\infty)$,
\begin{equation}\label{eq:L-T-of-hitting-time}
\mathbb{E}_z[e^{-\theta \zeta^-_\mathsf{a}}]=\frac{f_\theta(z)}{f_\theta(\mathsf{a})}.
\end{equation}
\end{lemma}
\begin{proof}
By Lemmas~\ref{proposition:h-to-f} and~\ref{lemma:f-theta-finite} the map $f_\theta$ is finite and $\mathscr{L}f_\theta=\theta f_\theta$, both on $(z^*,\infty)$. Since $h_\theta$ is not zero, $f_\theta$ is strictly positive everywhere. Besides, $Z_{\zeta^-_\mathsf{a}}=\mathsf{a}$ a.s.-$\PP_z$ on $\{\zeta^-_\mathsf{a}<\infty\}$,  because there are no negative jumps. By Theorem \ref{theorem:cbc-sde-contruction} and the non-explosiveness of $Z$, the process $(e^{-\theta (t\land \zeta^-_\mathsf{a})}f_\theta(Z_{t\land \zeta^-_\mathsf{a}}),t\geq 0)$ is a local martingale, which is bounded by $f_\theta(\mathsf{a})$ (since  $f_\theta$ is decreasing), hence a martingale. Therefore
\[\mathbb{E}_z\big[e^{-\theta (\zeta^-_\mathsf{a}\wedge t)}f_\theta(Z_{\zeta^-_\mathsf{a}\wedge t})\big]=f_\theta(z),\quad t\in [0,\infty).\]
Letting $t$ tend to $\infty$ gives the target identity \eqref{eq:L-T-of-hitting-time}.
\end{proof}
Uniqueness of the solution $h_\theta$  up to a multiplicative constant is settled by

\begin{lemma}\label{lemmauniquenessh} Assume that the \CBC does not explode. Then, up to a multiplicative constant, there is a unique nondecreasing, not zero, nonnegative function $h_\theta$, solution $h$ to $\mathscr{G}h=\theta h.$ 
\end{lemma}
\begin{proof}
Up to a multiplicative constant the function $f_\theta:(z^*,\infty)\to (0,\infty)$ satisfying $\mathbb{E}_z[e^{-\theta \zeta^-_\mathsf{a}}]=\frac{f_\theta(z)}{f_\theta(\mathsf{a})}$ for $\mathsf{a}\leq z$ from $(z^*,\infty)$ is unique evidently. In turn this guarantees the same kind of uniqueness of the nondecreasing, not zero, nonnnegative solution $h$ to $\mathscr{G}h=\theta h$, as if there were two different solutions, Lemma \ref{lemmaLT} would provide two different (in the preceding sense) functions $f_\theta$ (since finite Laplace transforms on a neighborhood of infinity determine continuous functions uniquely).
\end{proof}
All in all, under non-explosion of $Z$ the function $h_\theta$ of Lemma~\ref{proposition:h-to-f} exists uniquely (up to a multiplicative constant) and is strictly increasing and strictly positive everywhere. The proof of Theorem \ref{firstpassagetimestheorem}  follows straightforwardly by combining the above lemmas. 

\begin{remark}\label{Remark0notregular} If the existence of a strictly increasing solution $h:(0,\infty)\to (0,\infty)$ to  $\mathscr{G}h=\theta h$ is never in question, several (differing by more than a multiplicative constant) such solutions exist when the boundary $0$ of $\mathscr{G}$ is regular, see e.g. Borodin and Salminen \cite[Chapter~II, Section~1, Paragraph~10]{MR1912205}. Thus, when $Z$ does not explode, since there is a unique such solution  to $\mathscr{G}h=\theta h$, then the boundary $0$ of $V$ cannot be regular. At this stage though we cannot as yet specify whether the boundary $0$ is natural, entrance or exit, see the forthcoming Remark \ref{Remark0notexit}. Note however that under the assumption of Proposition \ref{propositionsuffcond} the process $Z$ does not explode and it can be checked from Feller's tests on the other hand that the boundary $0$ of $V$ is inaccessible in this case (hence either entrance or natural). 
\end{remark}
 The solution $h_ \theta$ may be represented with the help of $\tau_y$, the first hitting time of $y$ by the diffusion $V$. Namely we have, for $v<y$ from $(0,\infty)$,
\begin{equation}\label{LTV}
\mathbb{E}_v[e^{-\theta \tau_y}]=\frac{h_\theta(v)}{h_\theta(y)}.
\end{equation}
 Here, as usual, the subscript $v$ in the expectation indicates the starting value of $V$. 
\subsection{Extinction: proof of Theorem \ref{LTextinctiontheorem}}\label{secproofextinction}
We focus here on extinction under the assumption of non-explosion. We first verify  \eqref{idLT} in case $z^*=0$.
For sure $\zeta_{0+}:=\text{$\uparrow $-$\lim_{a\downarrow 0}$} \zeta_a^-\leq \zeta_0^-$. On $\{\zeta_{0+}=\infty\}$  trivially $\zeta_0^-=\infty=\zeta_{0+}$; on $\{\zeta_{0+}<\infty\}$, due to quasi-left continuity and the absence of negative jumps, a.s. $$Z_{\zeta_{0+}}=\lim_{\mathsf{a}\downarrow 0} Z_{\zeta_\mathsf{a}^-}=\lim_{\mathsf{a}\downarrow 0} \mathsf{a}=0,$$
and thus $\zeta_{0+}\geq \zeta_0^-$, which ensures that (again) $\zeta_{0+}=\zeta_0^-$. Hence, by letting $\mathsf{a}$ go to $0$ in \eqref{eq:L-T-of-hitting-time}, we have
\[\mathbb{E}_z[e^{-\theta \zeta_0^-}]=\frac{f_\theta(z)}{f_\theta(0+)}.\]
Besides, from \eqref{eq:from h-to-f}, $f_\theta(0+)=h_{\theta}(\infty)$. Therefore
\begin{align}\mathbb{E}_z[e^{-\theta \zeta_0^-}]&=\int_{0}^{\infty}ze^{-zx}\frac{h_\theta(x)}{h_\theta(\infty)}\ddr x. \label{explosionV=extinctionZ}
\end{align}
Thanks to \eqref{LTV}  we may indeed rewrite   this as 
\[\mathbb{E}_z[e^{-\theta \zeta_0^-}]=\mathbb{E}[e^{-\theta \tau_\infty^{\mathbbm{e}_z}}],\]
where $\tau_\infty^{\mathbbm{e}_z}$ is the explosion time of the diffusion $V$ started from an independent exponential random variable with parameter $z$. 

We proceed to  study accessibility of the boundary $0$ of the CBC. If $z^*=0$ then by letting $\theta$ go to $0$ in \eqref{explosionV=extinctionZ}, we see that it is accessible if and only if $\infty$ is accessible for the diffusion $V$. Recall the scale and speed measures of $V$, $S_V$ and $M_V$, given in \eqref{scalemeasure} and \eqref{speedmeasure} respectively.  Define
\begin{equation}\label{I}
\mathcal{I}:=\int^{\infty}_{x_0}S_V(x,\infty)\ddr M_V(x)=\int^\infty_{x_0} e^{Q(u)}\left(\int_u^{\infty}\frac{e^{-Q(x)}}{\Sigma(x)}\ddr x\right)\ddr u,  
\end{equation}
where $$Q(u):=\int_{x_0}^{u}\frac{\Psi(v)}{\Sigma(v)}\ddr v,\quad u\in (0,\infty).$$
Feller's classification ensures that $h_\theta(\infty)<\infty$ (i.e. $\infty$ accessible for $V$) if and only if $\mathcal{I}<\infty$ see e.g. \cite[Lemma 6.2, page 230]{zbMATH03736679}. We are left to show that $\mathcal{I}<\infty$ if and only if $\Psi(\infty)=\infty$ and $\int^{\infty}\frac{\ddr u}{\Psi(u)}<\infty$ (Grey's condition), indeed thanks to Proposition~\ref{propositionstatespace} this will handle also  \eqref{equiv}-\eqref{idLT} for the case when $z^*>0$.

Assume $\Psi(\infty)<\infty$ in the first instance, so that $-\Psi$ is the Laplace exponent of a subordinator. Since $-\Psi(v)\geq 0$ for all $v\geq 0$, we get the following lower bound, 
\label{page:computation-I}
\[\mathcal{I}=\int_{x_0}^{\infty} \left(\int_{u}^{\infty}\frac{e^{-\int_{u}^x\frac{\Psi(v)}{\Sigma(v)}\ddr v}}{\Sigma(x)}\ddr x\right)\ddr u\geq \int_{x_0}^{\infty} \left(\int_{u}^{\infty}\frac{\ddr x}{\Sigma(x)}\right)\ddr u= \int^{\infty}_{x_0}\frac{x-x_0}{\Sigma(x)}\ddr x=\infty,\]
where in the penultimate equality we have applied Tonelli's theorem, and we recall that $\Sigma(x)\leq \mathsf{C}x^2$ for all $x\in [x_0,\infty)$ for some constant $ \mathsf{C}<\infty$, which gives finally the divergence of the integral.

Now assume that $\Psi(\infty)=\infty$. Recall that $\Psi$ is positive increasing on $(\rho,\infty)$, where $\rho\in [0,\infty)$ is the largest zero of $\Psi$; moreover, $(0,\infty)\ni u\mapsto \Psi(u)/u$ is  nondecreasing. We may and do insist that $x_0\in (\rho,\infty)$. There exists $\mathsf{c}>0$ such that for all $u\in [x_0,\infty)$,  $\Psi(u)\geq \mathsf{c}u$. Then, for all $x\in[ x_0,\infty)$, $Q(x)=\int_{x_0}^{x}\frac{\Psi(u)}{\Sigma(u)}\ddr u\geq \mathsf{c}\int_{x_0}^{x}\frac{u}{\Sigma(u)}\ddr u$. 
 Therefore, $Q(x)\geq \frac{\mathsf{c}}{\mathsf{C}}\log x$ for all $x\in [x_0,\infty)$, in particular $Q(\infty)=\infty$. 

For typographical ease set also $$\varphi(u):=\int_u^{\infty}\frac{1}{\Sigma(x)}e^{-Q(x)}\ddr x\leq \int_u^{\infty}\frac{\ddr x}{x^{\mathsf{c}/\mathsf{C}}\Sigma(x)}<\infty,\quad u\in [x_0,\infty).$$ 
Then note that
\begin{align*}
\left(e^{Q(u)}\varphi(u)\right)'&=Q'(u)e^{Q(u)}\varphi(u)+e^{Q(u)}\varphi'(u)\\
&=Q'(u)e^{Q(u)}\varphi(u)-e^{Q(u)}\frac{1}{\Sigma(u)}e^{-Q(u)}\\
&=\big(\Psi(u)e^{Q(u)}\varphi(u)-1\big)\frac{1}{\Sigma(u)},\quad u\in [x_0,\infty).
\end{align*}
Hence, for $x\in [x_0,\infty)$, 
\begin{align}\label{twoterms}
\int_{x_0}^{x}\varphi(u)e^{Q(u)}\ddr u&=\int_{x_0}^{x}\frac{\ddr u}{\Psi(u)}+\int_{x_0}^{x}\left(e^{Q(u)}\varphi(u)\right)'\frac{\Sigma(u)}{\Psi(u)}\ddr u.
\end{align}
Furthermore, since $\Psi(u)\leq \Psi(x)$ for all $x\geq u\geq x_0$ and $Q(\infty)=\infty$, 
$$\Psi(u)e^{Q(u)}\varphi(u)=\Psi(u)e^{\int_{x_0}^{u}\frac{\Psi(v)}{\Sigma(v)}\dd v}\int_{u}^{\infty}\frac{1}{\Sigma(x)}e^{-\int_{x_0}^{x}\frac{\Psi(v)}{\Sigma(v)}\dd v}\ddr x\leq e^{\int_{x_0}^{u}\frac{\Psi(v)}{\Sigma(v)}\dd v}\int_{u}^{\infty}\frac{\Psi(x)}{\Sigma(x)}e^{-\int_{x_0}^{x}\frac{\Psi(v)}{\Sigma(v)}}\ddr x=1.$$
Thus $\left(e^{Q(u)}\varphi(u)\right)'\leq 0$ for all $u\in [x_0,\infty)$ and
\[\int_{x_0}^{x}\varphi(u)e^{Q(u)}\ddr u\leq \int_{x_0}^{x}\frac{\ddr u}{\Psi(u)},\quad x\in [x_0,\infty).\]
Hence, if Grey's condition holds, namely $\int_{x_0}^{\infty}\frac{\ddr u}{\Psi(u)}<\infty$, then $\mathcal{I}<\infty$ and the process $Z$ goes extinct. We now study the other direction of the equivalence and assume $\mathcal{I}<\infty$, the question being whether collisions can cause extinction in CBC processes for which Grey's condition is not fulfilled.
With $\mathsf{c}$ and $\mathsf{C}$ as above,  $\frac{\Sigma(u)}{\Psi(u)} \leq \frac{\mathsf{C}}{\mathsf{c}} u$ for all $u\in [x_0,\infty)$; by \eqref{twoterms},
\begin{align*}
\int_{x_0}^{x}\varphi(u)e^{Q(u)}\ddr u&\geq \int_{x_0}^{x}\frac{\ddr u}{\Psi(u)}+\int_{x_0}^{x}\left(e^{Q(u)}\varphi(u)\right)' \frac{\mathsf{C}}{\mathsf{c}} u \ddr u,\quad x\in [x_0,\infty).
\end{align*}
Via integration by parts:
$$\int_{x_0}^{x}\left(e^{Q(u)}\varphi(u)\right)' u \ddr u=\left[e^{Q(u)}\varphi(u)u  \right]_{u=x_0}^{u=x}-\int_{x_0}^{x}e^{Q(u)}\varphi(u)\ddr u,\quad x\in [x_0,\infty).$$
Combining the preceding two displays we get
\begin{align*}
(1+\mathsf{C}/\mathsf{c})\int_{x_0}^{x}\varphi(u)e^{Q(u)}\ddr u&\geq \int_{x_0}^{x}\frac{\ddr u}{\Psi(u)}-\frac{\mathsf{C}}{\mathsf{c}}\varphi(x_0) x_0  ,\quad x\in [x_0,\infty).
\end{align*}
Thus, on letting $x$ tend to $\infty$, by monotone convergence, if $\mathcal{I}<\infty$, i.e.  $\int_{x_0}^{\infty}\varphi(u)e^{Q(u)}\ddr u<\infty$,  then also $\int_{x_0}^{\infty}\frac{\ddr x}{\Psi(x)}<\infty$. \qed
\section{Laplace/Siegmund Duality and limiting distribution}\label{sectionduality}
\subsection{Laplace duality at the level of semigroups: proof of Proposition \ref{Laplacedualsemigroup}}
We start with a lemma ensuring that the diffusion $U$ does not explode, as was previously announced. Together with the assumed non-explosivity of $Z$ it will play a key role in establishing the Laplace duality. Recall $\mathscr{A}g=\Sigma g''-\Psi g'$ for $g\in C^2([0,\infty))$, the generator of $U$.
\begin{lemma}\label{lemma:U-not-explocive} 
The boundary $\infty$ of $U$ is inaccessible.
\end{lemma}
\begin{proof}
Consider the Feller test 
for the boundary $\infty$ of $U$ to be accessible. Set
\[\mathcal{J}:=\int_{x_0}^{\infty}\frac{\ddr x}{\Sigma(x)}\int_{x}^{\infty}\exp\left(\int_{x}^{y}\frac{\Psi(u)}{\Sigma(u)}\ddr u\right)\ddr y.\]
Then the boundary $\infty$ is inaccessible for $U$ if and only if $\mathcal{J}=\infty$. The non-subordinator case for which $\Psi>0$ in a neighbourhood of $\infty$ satisfies clearly $\mathcal{J}=\infty$. Assume now that $-\Psi$ is the Laplace exponent of a subordinator. Note that $\frac{\Psi(u)}{u}\underset{u\rightarrow \infty}{\longrightarrow} \textbf{d}\in (-\infty, 0 ]$. Let $\gamma\in (-\infty,\textbf{d})$; then there is a large enough $u_0\in [0,\infty)$ such that for all $u\in[ u_0,\infty)$ we have $\Psi(u)\geq  \gamma u$. Furthermore, since $\Sigma$ is convex, the map $(0,\infty)\ni u\mapsto \frac{\Sigma(u)}{u}$ is non-decreasing and thus for  $u\geq x$ from $(0,\infty)$ we get  $\frac{u}{\Sigma(u)}\leq \frac{x}{\Sigma(x)}$.
Hence, for $x\in [u_0,\infty)$, since $\gamma<0$,
\begin{align*}
\exp\left(\int_{x}^{y}\frac{\Psi(u)}{\Sigma(u)}\ddr u\right)\geq \exp\left(\int_{x}^{y}\frac{\gamma u}{\Sigma(u)}\ddr u\right)\geq \exp\left(\frac{\gamma x}{\Sigma(x)}(y-x)\right).
\end{align*}
Besides,
\[\int_{x}^{\infty}\exp\left(\frac{\gamma x}{\Sigma(x)}y\right)\ddr y=\frac{\Sigma(x)}{-\gamma x}\exp\left(\frac{\gamma x^2}{\Sigma(x)}\right)\] and
therefore
\begin{align*} \mathcal{J}&\geq\int_{x_0}^{\infty}\frac{\ddr x}{\Sigma(x)}\left(\int_{x}^{\infty}\exp\left(\frac{\gamma x}{\Sigma(x)}y\right)\ddr y\right) \exp\left(\frac{-\gamma x^2}{\Sigma(x)}\right)\\
&=\int_{x_0}^{\infty}\frac{\ddr x}{\Sigma(x)}\frac{\Sigma(x)}{-\gamma x}\exp\left(\frac{\gamma x^2}{\Sigma(x)}\right) \exp\left(\frac{-\gamma x^2}{\Sigma(x)}\right)\\
&=\int_{x_0}^{\infty}\frac{\ddr x}{-\gamma x}\ddr x=\infty,
\end{align*}
which concludes the argument.
\end{proof}
\noindent \textbf{Proof of Proposition \ref{Laplacedualsemigroup}}: We may and do assume $z>0$, $x>0$.  We work under a probability under which $Z$ and $U$ are independent processes starting at $z$ and $x$ respectively and apply the duality result of Ethier and Kurtz \cite[Corollary~4.4.15]{ethier}. Recall  $Z$ does not explode by assumption, while $U$ does not explode by Lemma~\ref{lemma:U-not-explocive}. 


 Let  $\mathsf{a}< x<\mathsf{b}$ be from  $(0,\infty)$ and put $\sigma_{\mathsf{a},\mathsf{b}}:=\sigma_\mathsf{a}^-\land \sigma_\mathsf{b}^+$, where $\sigma_\mathsf{c}^\pm:=\inf\{t\in [0,\infty):\pm	U_t\geq \pm \mathsf{c}\}$, $\mathsf{c}\in [0,\infty)$. In the notation of \cite[Corollaries~4.4.14~and~4.4.15]{ethier} take then $E_1:=E_2:=[0,\infty)$, $X:=Z$,  $Y:=U$, $\FF$ the natural filtration of $X$, $\GG$ the natural filtration of $Y$, $\alpha:=\beta:=0$, $\tau:=\infty$, $\sigma:=\sigma_{\mathsf{a},\mathsf{b}}$ and, for $(x',y')\in [0,\infty)\times [0,\infty)$,
\begin{align*}
f(x',y')&:=e^{-x'y'},\\
h(x',y')&:=g(x',y'):=\mathscr{L}_{x'}f(x',y')=({x'}^2\Sigma(y')+x'\Psi(y'))e^{-x'y'}=\mathscr{A}_{y'}f(x',y').
\end{align*}
Note that  $f$ is bounded. The function $g$ ($=h$) is bounded separately in each coordinate, but in general not globally; nevertheless it is bounded on $[0,\infty)\times [\mathsf{a},\mathsf{b}]$ (since $\Psi$ and $\Sigma$ are continuous on $(0,\infty)$ and the maps $[0,\infty)\ni u\mapsto ue^{-u}$ and $[0,\infty)\ni u\mapsto u^2e^{-u}$ are bounded), which is why we (have had to) employ $\sigma_{\mathsf{a},\mathsf{b}}$. Furthermore, from Corollary~\ref{corollaray:genZ} it follows that,  for each $y'\in [0,\infty)$, the process 
\begin{equation*}
\left(f(X_t,y')-\int_0^{t}g(X_s,y')\dd s,t\geq 0\right)
\end{equation*}
 is an $\FF$-martingale. On the other hand the process
 \begin{equation*}
\left(f(x',Y_{t\land \sigma_{\mathsf{a},\mathsf{b}}})-\int_0^{t\land \sigma_{\mathsf{a},\mathsf{b}}} h(x',Y_s)\dd s,t\geq 0\right)
\end{equation*}
is  a $\GG$-martingale for all $x'\in [0,\infty)$. Combining all of the preceding we infer that  the conditions of \cite[Corollary~4.15]{ethier} are met and we get that for all $t\in [0,\infty)$,\label{page:duality-computation}
\begin{align*}
&\EE_z[e^{-x Z_t}]-\EE_x[e^{-x U_{t\land \sigma_{\mathsf{a},\mathsf{b}}}}]=\EE[e^{-xX_t}]-\EE[e^{-zY_{t\land \sigma_{\mathsf{a},\mathsf{b}}}}]\\
&=\int_0^t\EE\left[\left(1-\mathbbm{1}_{[0,\sigma_{\mathsf{a},\mathsf{b}}]}(t-s)\right)e^{-X_{s}Y_{(t-s)\land \sigma_{\mathsf{a},\mathsf{b}}}}\left(X_{s}^2\Sigma(Y_{(t-s)\land \sigma_{\mathsf{a},\mathsf{b}}})+X_{s}\Psi(Y_{(t-s)\land \sigma_{\mathsf{a},\mathsf{b}}})\right)\right]\dd s\\
&=\int_0^t\EE\left[e^{-X_sY_{ \sigma_{\mathsf{a},\mathsf{b}}}}\left(X_s^2\Sigma(Y_{ \sigma_{\mathsf{a},\mathsf{b}}})+X_s\Psi(Y_{ \sigma_{\mathsf{a},\mathsf{b}}})\right);\sigma_{\mathsf{a},\mathsf{b}}<t-s\right]\dd s\\
&=\EE\left[\int_0^{t-\sigma_{\mathsf{a},\mathsf{b}}\land t}e^{-Z_sU_{ \sigma_{\mathsf{a},\mathsf{b}}}}\left(Z_s^2\Sigma(U_{ \sigma_{\mathsf{a},\mathsf{b}}})+Z_s\Psi(U_{ \sigma_{\mathsf{a},\mathsf{b}}})\right)\dd s\right]\\
&=\EE[e^{-\mathsf{a}Z_{t-\sigma_\mathsf{a}^-}}-e^{-\mathsf{a}Z_0};\sigma_\mathsf{a}^-\leq\sigma_\mathsf{b}^+\land t]+\EE[e^{-\mathsf{b}Z_{t-\sigma_\mathsf{b}^+}}-e^{-\mathsf{b}Z_0};\sigma_\mathsf{b}^+\leq\sigma_\mathsf{a}^-\land t],
\end{align*}
where the last equality follows from the constancy in time of the expectation of the $\FF$-martingales noted above and by independence of $U$ from $Z$. In the preceding display we may now let $\mathsf{b}\to\infty$ and $\mathsf{a}\downarrow 0$ (in this order) and using the fact that neither $Z$ (for the first term) nor $U$ (for the second term)  explode, we get the Laplace duality between the non-explosive \CBC process $Z$ and the minimal diffusion $U$ (absorbed at $0$):
\begin{equation}\label{dualidentityproof} \mathbb{E}_z[e^{-xZ_t}]=\mathbb{E}_{x}[e^{-zU_{t}}]. \qed
\end{equation} 

\begin{remark}
Assume $Z$ does not explode. Laplace duality of $Z$ with $U$ on the level of the semigroups entails  that the set $\mathcal{C}\cap C_0([0,\infty))$ of Example~\ref{example:in-domain-genZ} is invariant (in the sense that it is closed under the action  of the semigroup of $Z$); it is also dense in $C_0([0,\infty))$ by the ``locally compact'' version of Stone-Weierstrass. Therefore, see e.g. Kallenberg \cite[Proposition~19.9]{Kallenberg}, it is a core for the infinitesimal generator of the Feller process $Z$ on  $C_0([0,\infty))$. 
\end{remark}

\subsection{Siegmund duality: proof of Proposition~\ref{Siegmundualsemigroup}}
Recall that $V$ is the (minimal) diffusion with generator $\mathscr{G}$. The assumptions $S_V(0,x_0]=\infty$ and $S_V(x_0,\infty)=\infty$, which are in effect, entail that $V$ has boundaries $0$ and $\infty$ inaccessible: in other words either entrance or natural. An application of \cite[Theorem 6.1]{foucart2021local} ensures that the Siegmund dual process $U$ of $V$, i.e. the process such that for all $x,y\in (0,\infty)$ and $t\geq 0$:
\begin{equation}\label{Siegmundinlemma}\mathbb{P}_x(U_t<y)=\mathbb{P}_y(V_t>x)\end{equation}
is indeed our diffusion with generator $\mathscr{A}$.   Finally, we know that under the assumption $S_V(0,x_0]=\infty$  the process $Z$ does not explode. Therefore, applying Proposition \ref{Laplacedualsemigroup}, introducing an exponential random variable $\mathbbm{e}_z$ with parameter $z$ independent of $U$, and plugging it into \eqref{Siegmundinlemma}, we get 
\begin{equation}\label{lastidentitybeforeinvariant} \mathbb{E}_z[e^{-xZ_t}]=\mathbb{E}_x[e^{-zU_t}]=\mathbb{P}_x(\mathbbm{e}_z>U_t)=\int_{0}^{\infty}ze^{-zy}\mathbb{P}_y(V_t>x)\ddr y. \qed
\end{equation}
\begin{remark}\label{Remark0notexit} 
Siegmund duality exchanges the nature of the boundaries (the scale and speed measures are interchanged), see \cite[Table 7]{foucart2021local}.  
We observed in Remark~\ref{Remark0notregular} that, under the assumption of non-explosion of $Z$, the boundary $0$ of $V$ is not regular, and noted that it can therefore be either natural, entrance or exit. The exit option is precluded, since if $Z$ does not explode and $0$ is an exit for $V$, then by Siegmund duality, $0$ is an entrance for $U$, and letting $x$ tend to $0$ in the Laplace duality \eqref{dualidentityproof} yields $\mathbb{P}_z(Z_t<\infty)=\mathbb{E}_{0+}[e^{-zU_t}]<1$, which contradicts the non-explosivity of $Z$.  Together with Remark~\ref{Remark0notregular}, it establishes in fact that inaccessibility of $\infty$ for $Z$ automatically entails inaccessibility of $0$ for $V$. Establishing whether the latter is even an equivalence does not seem to follow easily from our approach. We have also seen in Lemma~\ref{lemma:U-not-explocive} that the boundary $\infty$ of $U$ is either natural or entrance. By Siegmund duality, $V$ has therefore its boundary $\infty$ either natural or exit. In particular, it is important to note that under the assumption of non-explosion of $Z$ the diffusion with generator $\mathscr{G}$ is uniquely specified 
since neither one of its boundaries is regular.
\end{remark}
\subsection{Limiting distribution: proof of Theorem \ref{stationarydisttheorem}}
By assumption $S_V(0,x_0]=\infty$ and $S_V(x_0,\infty)=\infty$, which ensures that $V$ is a regular recurrent diffusion on $(0,\infty)$. In this setting, the only possible invariant measure for $V$ is its speed measure $M_V$.
If the latter is finite on $(0,\infty)$, the diffusion $V$ is positive recurrent and converges in distribution towards the law $\frac{M_V(\ddr v)}{M_V(0,\infty)}$, see e.g. Rogers and Williams \cite[Theorem 54.5, page 303]{zbMATH01515832}. Under this proviso we see by letting $t$ tend to $\infty$ in \eqref{lastidentitybeforeinvariant} that for any $x\in [0,\infty)$ and $z\in (0,\infty)$:
$$\mathbb{E}_z[e^{-xZ_t}]\underset{t\rightarrow \infty}{\longrightarrow} \frac{M_V(x,\infty)}{M_V(0,\infty)}.$$
We now study the case in which $M_V$ gives an infinite mass to $(0,\infty)$. It is slightly easier to work directly with the dual diffusion $U$. Recall that, up to a multiplicative constant (we avoid making this reservation explicit below) $M_V=S_U$, where $S_U$ is the scale measure of $U$. The following three  cases may occur, see e.g. Karatzas and Shreve \cite[Proposition 5.22 page 345]{karatzas}.
\begin{enumerate}[(i)]
\item If $S_U(0,x_0]=M_V(0,x_0]<\infty$ and $S_U(x_0,\infty)=M_V(x_0,\infty)=\infty$, then, for all $x\in [0,\infty)$,
\[\mathbb{P}_x(\underset{t\rightarrow \infty}{\lim}\ U_t=0)=1;\]
hence, by \eqref{dualidentityproof},
$\underset{t\rightarrow \infty}{\lim}\  \mathbb{E}_z[e^{-xZ_t}]=1$ and $Z_t$ converges in probability towards $0$ as $t$ tends to infinity.
\item If $S_U(0,x_0]=M_V(0,x_0]=\infty$ and $S_U(x_0,\infty)=M_V(x_0,\infty)<\infty$, then, for all $x\in (0,\infty)$,
\[\mathbb{P}_x(\underset{t\rightarrow \infty}{\lim}\ U_t=\infty)=1;\]
hence, by \eqref{dualidentityproof},
$\underset{t\rightarrow \infty}{\lim}\  \mathbb{E}_z[e^{-xZ_t}]=0$ and $Z_t$ converges in probability towards $\infty$. 
\item  If $S_U(0,x_0]=M_V(0,x_0]=\infty$ and $S_U(x_0,\infty)=M_V(x_0,\infty)=\infty$ then $U$ is recurrent and
by the interchange of scale and speed measures the assumption $S_V(0,x_0]=\infty$ and $S_V(x_0,\infty)=\infty$ implies $M_U(0,x_0]=\infty$ and $M_U(x_0,\infty)=\infty$, where $M_U$ denotes the speed measure of $U$. We see therefore that $U$ is a null recurrent diffusion without a limiting distribution on $[0,\infty]$. A final application of \eqref{dualidentityproof} entails that $Z$ cannot have a limiting distribution. $\qed$
\end{enumerate}

\subsection{Characterizing CBCs: proof of Theorem~\ref{thm:laplace-converse}}
Recall that the usual meaning attached to  $Z$, $\mathscr{L}$, $\mathscr{A}$, $\Sigma$ and $\Psi$ in this paper is suspended in the context of Theorem~\ref{thm:laplace-converse}. 

We start by establishing a lemma of independent interest, specifying how the generator of a positive Feller process with no negative jumps may act on exponential functions.

\begin{lemma}[Courrège form on exponentials]\label{lemmacourrege} Assume that $Z$ is a positive Feller process with no negative jumps, $0$ absorbing, and infinitesimal generator $\mathscr{L}$, whose domain includes the  Schwartz space of rapidly decaying functions on $[0,\infty)$. For $x\in (0,\infty)$ let $e_x:=([0,\infty)\ni z\mapsto e^{-xz})$ be the exponential function (of rate $x$). Then, for any $f\in C^\infty_c([0,\infty))\cup\{e_x: x\in (0,\infty)\}$, the generator $\mathscr{L}$ acts on $f$ as follows: 
\begin{equation}\label{eq.courrege-extendedlemma}
 \mathscr{L}f(z)=\int\left(f(z+h)-f(z)-hf'(z)\mathbbm{1}_{(0,1]}(h)\right)\nu(z,\ddr h)+a(z)f''(z)+b(z)f'(z)-c(z)f(z),
 \end{equation}
with $\nu(z,\ddr h)$ a Lévy measure on $(0,\infty)$, $a(z)\in[ 0,\infty)$, $b(z)\in \mathbb{R}$, $c(z)\in [0,\infty)$ for $z\in [0,\infty)$ and $a(0)=b(0)=c(0)=\nu(0,\ddr h)=0$.
\end{lemma} 
\begin{proof}
Because $0$ is an absorbing state for $Z$ we may and for a moment do extend it to a Feller process on the whole real line by taking it as constant on $(-\infty,0)$. So extended, its infinitesimal generator includes $C^\infty_c(\mathbb{R})$. From the so-called Courrège form of $\mathscr{L}$, see e.g. B\"ottcher et al.   \cite[Theorem~2.21]{levy-matters-III}, it follows that for $f\in C^\infty_c(\mathbb{R})$,
 \begin{equation}\label{eq.courrege-extended}
 \mathscr{L}f(z)=\int\left(f(z+h)-f(z)-hf'(z)\mathbbm{1}_{\{\vert h\vert\leq 1\}}\right)\nu(z,\ddr h)+a(z)f''(z)+b(z)f'(z)-c(z)f(z)
 \end{equation}
 for certain L\'evy measures $\nu(z,\ddr h)$ on $\mathbb{R}$, diffusion coefficients $a(z)\in [0,\infty)$, drifts $b(z)\in \mathbb{R}$ and killing rates $c(z)\in [0,\infty)$ as $z$ runs over $\mathbb{R}$.   Because $Z$ has no negative jumps we know from applying  Dynkin's characteristic operator \cite[Theorem~1.39]{levy-matters-III} that, for all $z\in (0,\infty)$,  for all $f\in C^\infty_c(\mathbb{R})$ whose support is bounded away on the right from $z$, $\mathscr{L}f(z)=0$: 
the key is simply to note that, for all sufficiently small $r\in (0,\infty)$, at time\footnote{Notwithstanding the first paragraph of this subsection we use $\zeta_{z+r}^+$ and $\zeta_{z-r}^-$ in their  established relation, see  \eqref{eq.upwards-def} \& \eqref{eq.first-passage.def}, vis-\`a-vis the process $Z$.} $\zeta_{z+r}^{+}\wedge \zeta_{z-r}^-$ of first exit from the interval of radius $r$ around $z$ the process $Z$ is either above $z+r$ or equal to $z-r$ by the absence of negative jumps, hence $f(Z_{\zeta_{z+r}^{+}\wedge \zeta_{z-r}^-})=0=f(z)$ a.s.. Therefore, for $z\in (0,\infty)$, $\nu(z,\ddr h)$ is carried by $(0,\infty)$. 

 We now revert back to the non-extended process and observe that \eqref{eq.courrege-extended} then holds true for $f\in C^\infty_c([0,\infty))$. The action of its right-hand side extends naturally to all $f\in C^2_b([0,\infty))$ and we use the symbol $\tilde{\mathscr{L}}$ for the corresponding operator. In short, 
 \begin{equation}\label{eq:L-tilde-L-Ccinf}
 \mathscr{L}=\tilde{\mathscr{L}}\text{ on }C^\infty_c([0,\infty)),
 \end{equation}
  the action of $\tilde{\mathscr{L}}$ on $C^2_b([0,\infty))$ being given by the right-hand side of \eqref{eq.courrege-extended}. Since $0$ is absorbing for $Z$ we may and do take $a(0)=b(0)=c(0)=\nu(0,\dd h)=0$. 

Next, fix $x\in (0,\infty)$ and we show that the equality in \eqref{eq:L-tilde-L-Ccinf} extends also to the exponential map $e_x$. Let $(\phi_n)_{n\in \mathbb{N}}$ be a sequence  in  $C^\infty_c([0,\infty))$ satisfying $\mathbbm{1}_{[0,n]}\leq \phi_n\leq 1$ for all $n\in \mathbb{N}$ and $\lim_{n\to\infty}\phi_n=1$ pointwise (such  smooth transition functions exist, see e.g. \cite[page 49]{conlon}). Then, on the one hand, it is clear by bounded convergence that $\tilde{\mathscr{L}}(e_x\phi_n)\to \tilde{\mathscr{L}}e_x$ pointwise as $n\to\infty$. On the other hand, for all $z\in [0,\infty)$, for all $n\in \mathbb{N}$ with $n\geq z+1$, $$\vert \mathscr{L}(\phi_ne_x)-\mathscr{L}e_x\vert(z)\leq \limsup_{t\downarrow 0}\frac{\EE_z[e^{-xZ_t};Z_t>n]}{t}\leq e^{-xn/2}\lim_{t\downarrow 0}\frac{\EE_z[f(Z_t)]}{t}=e^{-xn/2}\mathscr{L}f(z)$$  by choosing any $f$ from the Schwartz space of rapidly decaying functions on $[0,\infty)$ that majorizes $e_{x/2}\mathbbm{1}_{[z+1,\infty)}$  but vanishes at $z$ (it exists, it may depend on $x$ and $z$, but not on $n$; one way to get it is by multiplying $e_{x/2}$ with a  smooth transition function that vanishes on $[0,z+1/4)$ and is equal to one on $[z+3/4,\infty)$). Letting $n\to\infty$ in the preceding display we deduce that also $\mathscr{L}(e_x\phi_n)\to \mathscr{L}e_x$ pointwise as $n\to\infty$. But $\mathscr{L}(e_x\phi_n)=\tilde{\mathscr{L}}(e_x\phi_n)$ for all $n\in \mathbb{N}$. Hence $\mathscr{L}e_x=\tilde{\mathscr{L}}e_x$. 
\end{proof}

\begin{lemma} Under the assumptions of Lemma \ref{lemmacourrege}, if the generator $\mathscr{L}$ further satisfies the Laplace duality relationship 
\begin{equation}\label{laplacedualitylemma}\mathscr{L}_z e^{-xz}=\mathscr{A}_x e^{-xz}:=z^{2}e^{-xz}\Sigma(x)+ze^{-xz}\Psi(x),\quad\{x,z\}\subset [0,\infty),
\end{equation}
with the generator $\mathscr{A}$ of a one-dimensional diffusion on $[0,\infty)$ having drift $-\Psi$ and non-zero diffusion coefficient $\Sigma$, both assumed to be continuous at $0$, then $\Psi$ and $\Sigma$ are Lévy-Khintchine functions of the form \eqref{branchingmechanism} and \eqref{collisionmechanism}, respectively. 
\end{lemma}
\begin{proof}
Since $\mathscr{L}1=0$ (recall the comments following Theorem~\ref{thm:laplace-converse}),  from \eqref{laplacedualitylemma}, on setting $x=0$ we get $\Sigma(0)=0=\Psi(0)$. Furthermore, \eqref{laplacedualitylemma} and \eqref{eq.courrege-extendedlemma} tell us that for $\{x,z\}\subset ( 0,\infty)$, 
\begin{align}\label{eq:gen-evaluated}
\nonumber z^2\Sigma(x)+z\Psi(x)&=e^{xz}\mathscr{A}_xe^{-xz}=e^{xz}\mathscr{L}_ze^{-xz}\\
&=\int_{0}^{\infty}(e^{-xh}-1+xh\mathbbm{1}_{\{h\leq 1\}})\nu(z,\ddr h)-b(z)x+a(z)x^2-c(z).
\end{align} 
Letting $x\downarrow 0$ renders $c(z)=0$ by continuity of $\Sigma$ and $\Psi$ at $0$. 
Then  dividing by $z$ we get
\begin{equation}\label{divbyz} z\Sigma(x)+\Psi(x)=\int_{0}^{\infty}(e^{-xh}-1+xh\mathbbm{1}_{\{h\leq 1\}})z^{-1}\nu(z,\ddr h)-z^{-1}b(z)x+z^{-1}a(z)x^2.
\end{equation}
For any fixed $z$ the right-hand side  of \eqref{divbyz} is analytic in $x\in \mathbb{C}$ with $\Re(x)> 0$ and continuous in $x\in \mathbb{C}$ with $\Re(x)\geq 0$. Considering two different $z$ we deduce that $\Sigma$ and $\Psi$ admit unique extensions to continuous maps defined on $\{\Re\geq 0\}$, analytic on $\{\Re>0\}$. By analytic continuation and continuity at fixed $x$ \eqref{divbyz} then obtains for all $x\in \mathbb{C}$ with $\Re(x)\geq 0$,  for imaginary $x$ in particular.



It now follows that the characteristic functions of the infinitely divisible distributions whose Laplace exponents are given by the right-hand sides of \eqref{divbyz}  converge weakly as $z\downarrow 0$ towards a continuous function, namely $(\mathbb{R}\ni x\mapsto e^{\Psi(-\mathrm{i} x)})$. By L\'evy's continuity theorem it implies weak convergence of the infinitely divisible distributions, and since the latter are sequentially closed under weak convergence, Sato \cite[Lemma~2.7.8]{MR3185174}, we deduce that the limiting distribution is itself infinitely divisible, i.e. $(\mathbb{R}\ni x\mapsto \Psi(-\mathrm{i} x))$ is the characteristic exponent of a L\'evy process. We also know that convergence of characteristic exponents of L\'evy processes implies their weak convergence for the Skorohod topology, see e.g. Jacod and Shiryaev \cite[Corollary~VII.3.6]{jacod1987limit}. This allows us to infer \cite[Corollary~VI.2.8]{jacod1987limit} finally that $\Psi$ is   L\'evy-Khintchine of the spectrally positive type, i.e. it takes the form \eqref{branchingmechanism}. 

%
Similarly, dividing again by $z$ in \eqref{divbyz} we get
\begin{equation}\label{divbyz2} \Sigma(x)+z^{-1}\Psi(x)=\int_{0}^{\infty}(e^{-xh}-1+xh\mathbbm{1}_{\{h\leq 1\}})z^{-2}\nu(z,\ddr h)-z^{-2}b(z)x+z^{-2}a(z)x^2.
\end{equation}
Applying the very same reasoning but this time with $z\to\infty$ in lieu of $z\downarrow 0$ allows  to conclude that $\Sigma$ is  Lévy-Khintchine of the spectrally positive type as per \eqref{collisionmechanism}.
\end{proof}
\noindent \textbf{Proof of Theorem~\ref{thm:laplace-converse}}: Returning now to \eqref{eq:L-tilde-L-Ccinf} \& \eqref{eq:gen-evaluated}, since a L\'evy-Khintchine function determines the associated L\'evy triplet uniquely, it follows that $\mathscr{L}$ acts on $C_ c^\infty([0,\infty))$ according to  \eqref{genZ}.\qed


\vspace{0.5cm}
\noindent\textbf{Acknowledgements.} Ren\'e Schilling has the authors' thanks for a stimulating exchange on the topic of the Courr\`ege form of the infinitesimal generators of Feller processes.  MV acknowledges support from the Slovenian Research Agency (project No. N1-0174 and programme No. P1-0402).

\bibliography{doku}
\bibliographystyle{plain}

\end{document}